\documentclass [11pt,oneside,a4paper,mathscr]{amsart}
\usepackage[titletoc]{appendix}
\usepackage{mathrsfs}
\usepackage{csquotes}
\usepackage{geometry,braket}
\newgeometry{margin=1.3in}
\usepackage[utf8]{inputenc}
\usepackage{amsfonts, amsmath, amssymb, amsthm,mathtools, stmaryrd, blkarray,bm,ytableau, mathrsfs} 
\usepackage[backend=biber, style=ext-numeric, sorting=nyt,maxbibnames=99]{biblatex}
\usepackage{tikz}
\usepackage{tikz-3dplot}
\usepackage{tikz-cd}
\usepackage[mathscr]{euscript}
\renewbibmacro{in:}{}
\DeclareFieldFormat[article]{volume}{\mkbibbold{#1}}
\DeclareMathAlphabet{\mathpzc}{OT1}{pzc}{m}{it}

\DeclareDelimFormat[bib,biblist]{nametitledelim}{\addcolon\space}

\DeclareFieldFormat[article,inbook,incollection,inproceedings,patent,thesis,unpublished]{title}{#1\isdot}

\DeclareFieldFormat{date}{\mkbibparens{#1}}
\DeclareFieldFormat[article,periodical,inbook]{date}{#1}
\DeclareFieldFormat[article,periodical,inbook]{title}{\textit{#1}}
\DeclareFieldFormat[article,periodical]{journaltitle}{#1}
\DeclareFieldFormat[inbook]{booktitle}{#1}
\renewbibmacro*{volume+number+eid}{%
	\printfield{volume}%
	\setunit*{\addnbspace}
	\printfield{number}%
	\setunit{\addcomma\space}%
	\printfield{eid}}%
\DeclareFieldFormat[article]{number}{\mkbibparens{#1}}
\DeclareFieldFormat{pages}{#1}

\usepackage{verbatim}
\usepackage[normalem]{ulem}
\usepackage{tikz-cd}
\usepackage{enumitem}
\usepackage{array}
\usepackage{hyperref}

\theoremstyle{definition}
\newtheorem{theorem}{Theorem}[section]
\newtheorem{theorema}{Theorem}[section] 
\newtheorem{corollary}[theorem]{Corollary}
\newtheorem{lemma}[theorem]{Lemma}
\newtheorem{proposition}[theorem]{Proposition}

\numberwithin{equation}{section}

\theoremstyle{definition}
\newtheorem{definition}[theorem]{Definition}
\newtheorem{example}[theorem]{Example}
\newtheorem{remark}[theorem]{Remark}

\theoremstyle{remark}

\DeclareFontFamily{OT1}{rsfs}{}
\DeclareFontShape{OT1}{rsfs}{n}{it}{<-> rsfs10}{}
\DeclareMathAlphabet{\curly}{OT1}{rsfs}{n}{it}

\newcommand{\cC}{\mathcal{C}}

\newcommand{\cI}{\mathcal{I}}
\newcommand{\cJ}{\mathcal{J}}

\newcommand{\cM}{\mathcal{M}}

\newcommand{\cS}{\mathcal{S}}

\newcommand{\cX}{\mathcal{X}}
\newcommand{\cY}{\mathcal{Y}}

\allowdisplaybreaks
\begin{document}
	\baselineskip=14.5pt
	\title{Grothendieck topologies with logarithmic modifications}
	\author{Xianyu Hu}
	\address{Fakult\"at f\"ur Mathematik,
	Technische Universit\"at M\"unchen, D-85747 Garching bei M\"unchen, Germany}
	\email{xianyu.hu@tum.de}
	\author{Maximilian Schimpf}
	\address{Max Planck Institute for Mathematics, Bonn, Germany}
	\email{mschimpf@mpim-bonn.mpg.de}
	\date{\today}	
\begin{abstract}
Many concepts in log geometry are invariant under log blow-ups. To formalize this invariance, we introduce the m-open, m-\'etale, m-smooth, m-fppf, and m-fpqc topologies for fs log schemes. These refine the standard topologies from scheme theory by treating abstract log modifications as covers. For example, the m-\'etale topology is a subtopology of full log \'etale topology, characterized by a stronger lifting property than for log \'etale maps. Along the way, we identify and correct an error in the definition of the full log \'etale topology. We also prove a global integralization theorem by logarithmic blow-ups and use it to describe the m-open topos as a limit over log blow-ups. Finally, we characterize the sheaves on the m-type sites and connect the m-open site to Kato's valuative space.
\end{abstract}

\maketitle
\setcounter{tocdepth}{2}
\tableofcontents
\maketitle
\section*{Introduction}
Log modifications are fundamental to log geometry: log Gromov-Witten theory, log Chow theory and log Hochschild (co)homology are all 
invariant under log modifications (see, e.g., \cite{abramovich2020decomposition, Punctured-log-maps, abramovich2017boundedness, abramovich2018birational, barrott2020logarithmicchowtheory, bousseau2024stable, CAVALIERI_CHAN_ULIRSCH_WISE_2020, Schmitt2025logarithmic, Hodge-bundle,  Molcho_Wise_2022, ranganathan2022logarithmic}). Moreover, the log Quot spaces of \cite{kennedy2023logarithmic,kennedy2025logarithmic,maulik2024logarithmic,maulik2025logarithmic} parametrize sheaves on log modifications. As hinted in \cite{kennedy2025logarithmic}, these should perhaps more naturally be interpreted as sheaves on an enlargement of the Zariski (resp. \'etale, smooth, flat) topology obtained by declaring log modifications to be covers. This paper formalizes that intuition by systematically developing what we call the \emph{m-type topologies}.

\subsection{M-type morphisms} To formalize the m-type topologies, we will introduce \emph{m-type morphisms}. First, we have the \emph{m-open morphisms}:
\begin{definition}\label{m-open morphisms}
A morphism $X\rightarrow Y$ of fs log schemes is called \emph{m-open} if one can cover $X$ by opens $U\subset X$ so that $U\hookrightarrow X\rightarrow Y$ is a log \'etale monomorphism.
\end{definition}
We view m-open morphisms as a combination of strict open immersions and \emph{abstract log modifications} (see Definition \ref{defi: log mod} and Remark \ref{m-open local composition} for details).
Moreover, it turns out that m-open morphisms between fs log schemes with Zariski log structure can be described by an explicit chart criterion (see Proposition \ref{chart for m-open}).
This motivates our definition of \emph{m-\'etale}, \emph{m-smooth}, and \emph{m-flat} morphisms:
\begin{definition}\label{Definition of m-type morphism}
A morphism \( f: X \to Y \) of fs log schemes is called \emph{m-\'etale} (resp. \emph{m-smooth}, \emph{m-flat}) if for any $x \in X$ and any strict \'etale neighborhood $V \to Y$ of $f(x)$ admitting a chart $Q \to \Gamma(V, \mathcal{M}_Y)$, there exists a strict \'etale neighborhood $U\rightarrow X$ of $x$ such that $U\rightarrow X\rightarrow Y$ factors over a morphism $U\rightarrow V$ and admits a morphism of charts
    \[
    \begin{tikzcd}
        Q \arrow[r] \arrow[d, "\theta"'] & \Gamma(V, \mathcal{M}_{Y}) \arrow[d] \\
        P \arrow[r] & \Gamma(U, \mathcal{M}_X)
    \end{tikzcd}
    \]
    such that:
    \begin{enumerate}
        \item The morphism $\theta$ induces an isomorphism $Q^{\mathrm{gp}} \cong P^{\mathrm{gp}}$.
        \item The induced morphism $U \to V \times_{\mathbb{A}_Q} \mathbb{A}_P$ is strict \'etale (resp. smooth, flat).
    \end{enumerate}
\end{definition}
Similarly, a morphism is m-\'etale (resp. m-smooth, m-flat) if and only if it admits a strict \'etale local factorization as a strict \'etale (resp. smooth, flat) morphism and a log blow-up (see Proposition \ref{basic properties of m-type morphisms} for a precise statement). 

Rather surprisingly, m-\'etale and m-smooth morphisms can also be characterized by an infinitesimal lifting criterion.
\begin{definition}\label{Definition of formally m-etale smooth morphism}
Let $f:X\rightarrow Y$ be a morphism of fs log schemes. We say $f$ is \emph{formally m-\'etale} (resp. \emph{m-smooth}) if for any given outer commutative square
\begin{equation}
\begin{tikzcd}
T' \arrow[d, "i", hook] \arrow[r] & X \arrow[d, "f"] \\
T \arrow[r] \arrow[ru, dashed]    & Y               
\end{tikzcd}
\end{equation}
so that
\begin{enumerate}
    \item the underlying morphism of schemes $\underline{i}\colon\underline{T'}\hookrightarrow \underline{T}$ is a first-order thickening over $Y$;
    \item the morphism $i\colon T'\hookrightarrow T$ is exact, i.e. the diagram 
\begin{equation}
\begin{tikzcd}
i^{*}\mathcal{M}_T \arrow[r] \arrow[d] & \mathcal{M}_T' \arrow[d] \\
i^{*}\mathcal{M}_T^{gp} \arrow[r]      & \mathcal{M}_{T'}^{gp}   
\end{tikzcd}
\end{equation}
is cartesian,
\end{enumerate} 
there exists, strict \'etale locally in $T$, \emph{exactly one} (resp. \emph{at least one}) dotted arrow making the diagram commute.
\end{definition}
Note that this is a strengthening of the classical definition of log \'etale and log smooth (see \cite[Sections 3.2 and 3.3]{logpaperbyKazuyaKato}); the difference being that the thickenings $i\colon T'\hookrightarrow T$ are not assumed to be strict. In analogy with (log) \'etale morphisms, we have:
\begin{theorema}\label{main theorem A}
    A morphism of fs log schemes is m-\'etale (resp. m-smooth) if and only if it is formally m-\'etale (resp. m-smooth) and locally of finite presentation.
\end{theorema}

\subsection{M-type topologies}
M-type morphisms naturally yield corresponding Grothendieck topologies:

\begin{definition}\label{definition of m-type coverings}
For each $\tau \in \{ \text{m-open}, \text{m-\'etale}, \text{m-smooth}, \text{m-fppf}, \text{m-fpqc}\}$, a family of morphisms of fs log schemes $\{ U_i \to U\}_{i \in I}$ is called a $\tau\textit{-covering}$ if and only if 
$f\colon\coprod_{i \in I} U_i \rightarrow U$
is universally surjective and 
\begin{enumerate}
    \item\label{item: stuff1} for $\tau\in \{\text{m-open}, \text{m-\'etale}, \text{m-smooth}\}$, we require that $f$ is of type $\tau$.
    \item\label{item: stuff2} for $\tau = \text{m-fppf}$, we require that $f$ is m-flat and locally of finite presentation.
    \item\label{item: stuff3} for $\tau = \text{m-fpqc}$, we require that $f$ is m-flat and for any quasi-compact open subset $V\subset U$ there is a quasi-compact open $W\subset \coprod_{i\in I} U_i$ so that the restriction $W\rightarrow V$ of $f$ is universally surjective.
\end{enumerate}
\end{definition}
We establish in Lemma \ref{lemma: m-type tops are tops} that these indeed satisfy all of the axioms of being a Grothendieck topology. In particular, we use a non-standard definition of universal surjectivity here 
(see Section \ref{subsection in introduction about universal surjectivity} for further discussion).
These topologies also come with associated sites:
\begin{definition}\label{definition of m-type sites}
Let $S$ be an fs log scheme. For each $\tau \in \{ \text{m-open}, \text{m-\'etale}, \text{m-smooth},\\ \text{m-fppf},\text{m-fpqc}\}$ and each $\tau'\in  \{ \text{m-open}, \text{m-\'etale}\}$, we define 
\begin{enumerate}
\item the \emph{big} $\tau\textit{-site}$ $(\mathsf{fsLogSch}/S)_\tau$ of $S$ as the category $\mathsf{fsLogSch}/S$ equipped with the $\tau$-topology.
\item the \emph{small} $\tau'\textit{-site}$ $S_{\tau'}$ of $S$ as the full subsite of $(\mathsf{fsLogSch}/S)_\tau$ spanned by those objects that are of type $\tau'$ over $S$. 
\end{enumerate}
\end{definition}
In analogy to classical scheme theory, we will not consider small sites for the other topologies since they are not functorial, see Remark \ref{remark on functoriality}.

The following theorem characterizes sheaves on our sites:
\begin{theorema}\label{thm:sheaf_conditions}
Let $S$ be an fs log scheme. For any $\tau \in \{ \text{m-open}, \text{m-\'etale}, \text{m-smooth}, \text{m-fppf},\\\text{m-fpqc}\}$ (resp. $\tau'\in  \{ \text{m-open}, \text{m-\'etale}\}$), let $\mathcal{F}$ be a presheaf on $(\mathsf{fsLogSch}/S)_{\tau}$ (resp. $S_{\tau'}$) valued in a category $\mathcal{C}$ with all limits. Then $\mathcal{F}$ is a sheaf if and only if
\begin{enumerate}
    \item for any strict $\tau\text{-cover}$ ($\tau'\text{-cover}$)
    $\{U_i \to U\}_{i \in I}$ in $(\mathsf{fsLogSch}/S)_{\tau}$ (resp.\ $S_{\tau'}$), we have  
    \[
    \mathcal{F}(U) = \mathrm{Eq}\left( \prod_{i \in I} \mathcal{F}(U_i) \rightrightarrows \prod_{i,j \in I} \mathcal{F}(U_i \times_U U_j) \right);
    \]
    \item restriction along a log blow-up $U \to V$ in $(\mathsf{fsLogSch}/S)_{\tau}$ (resp.\ $S_{\tau'}$) induces $\mathcal{F}(V) = \mathcal{F}(U)$.
\end{enumerate}
\end{theorema}
\subsection{Subtleties on universal surjectivity}\label{subsection in introduction about universal surjectivity}
A morphism \( f\colon X \rightarrow Y \) of fs log schemes is usually called \emph{universally surjective} if all of its pullbacks \( f_T\colon X \times_Y T \rightarrow T \) along morphisms of fs log schemes \( T \rightarrow Y \) are surjective. 

This notion is well-known from the study of the full log \'etale topology, whose covers are defined to be jointly universally surjective families of log \'etale morphisms \cite{original-log-etale-topology-theory-Fujiwara-Kato, Overview-of-the-work-Kato-School, NAKAYAMALOGETALECohomology2}.
However, in Section \ref{section where shit goes down} we point out errors in the cited works and show that the above notion of universal surjectivity actually leads to bad behavior of the log \'etale site. For example:
\begin{theorema}\label{thm: shit be fucked}
    For any field $k$, the scheme $\mathbb{A}^2_{k}$ equipped with its canonical toric log structure is \emph{not} quasi-compact in the log \'etale site. 
\end{theorema}
We will deduce this by showing that the collection 
    \[
        \left\{\left(\text{Bl}_I\mathbb{A}^2_k \right)_{\text{non-fix}}\rightarrow \mathbb{A}^2_k \right\}_{\emptyset\neq I\subset\mathbb{N}^2 \text{ ideal}}
    \]
of all log blow-ups restricted to the non-torus-fixed points forms a log \'etale cover that does not admit a finite subcover. 

Theorem \ref{thm: shit be fucked} contradicts \cite[Lemma 3.14]{NAKAYAMALOGETALECohomology2}, which claimed that a log scheme is quasi-compact in the log \'etale site if and only if its underlying scheme is quasi-compact. This invalidates much of \cite{NAKAYAMALOGETALECohomology2}, as this claim is heavily used throughout the paper. In particular, we explain in Remark \ref{remark: topology not generated} that, in contradiction to \cite[Proposition 3.9]{NAKAYAMALOGETALECohomology2}, the log \'etale topology is not generated by Kummer \'etale covers and log blow-ups. 

However, in Proposition \ref{prop: quasi-compact is quasi-compact}, we show that \cite[Lemma 3.14]{NAKAYAMALOGETALECohomology2} and hence the rest of \cite{NAKAYAMALOGETALECohomology2} hold if one instead defines log \'etale covers using the following notion of universal surjectivity, which we will use for the rest of our paper:
\begin{definition}\label{definition: universally surjective}
    A morphism $f\colon X\rightarrow Y$ of fs log schemes is \emph{universally surjective} if for any morphism $T\rightarrow Y$ from an \emph{integral} log scheme $T$, the base change $f_T\colon X\times_Y T\rightarrow T$ is surjective.
\end{definition}
This definition agrees with the notion of universal surjectivity that is implicit in \cite{Molcho_Wise_2022}, where the authors work in the category of integral log schemes throughout. Although, in principle, these issues affect all papers that use the log \'etale topology, we expect that all results hold if this change of terminology is made. 

In order to avoid confusion, we use \emph{weak universal surjectivity} to refer to the former definition of universal surjectivity at the beginning of this section.
\subsection{Global integralization by logarithmic blow-ups}
A further main result of this paper is a global integralization theorem by logarithmic blow-ups. We state this result in greater generality than is necessary for this paper as it is also central to \cite[§1]{coherentsheavesinloggyloggsterofdoom} and for the sake of future reference. 
As opposed to other integralization results of this kind, which appear in several forms in the literature (see e.g. \cite{wsrc0,abramovich2020relativedesingularizationprincipalizationideals,herr2025loggeometryliftingrational,quasi-unipotent,kato2022integral,molcho2019universalweaksemistablereduction}), the statement is fully global and does not assume log regularity or the existence of an Artin fan.
\begin{theorema}\label{lemma-log-blow-up-integral}
Let \(f : \mathcal X \to \mathcal Y\) be a morphism of algebraic stacks with fs log structures, and assume that \(\mathcal X\) is quasi-compact. Then there exists a coherent, nowhere empty log ideal \(\mathcal I \subset M_{\mathcal Y}\) such that \(\operatorname{Bl}_{\mathcal I}\mathcal Y\) has free log structure and the base change
\[
\operatorname{Bl}_{f^*\mathcal I}\mathcal X
=
\mathcal X \times_{\mathcal Y} \operatorname{Bl}_{\mathcal I}\mathcal Y
\longrightarrow
\operatorname{Bl}_{\mathcal I}\mathcal Y
\]
is integral.   
\end{theorema}
Here, we call an fs log structure on an algebraic stack $\cX$ \emph{free} if it can be covered by strict fppf neighborhoods that are charted by free monoids. Using this and Proposition \ref{m-type of Kummer is strict}, we obtain:
\begin{corollary}\label{corollary: m can be made non-m by log blowy}
    Let $X\rightarrow Y$ be an m-open (resp. m-\'etale, m-smooth, m-flat) morphism of fs log schemes with $X$ quasi-compact. Then there is a log blow-up $\widetilde{Y}\rightarrow Y$ so that the base change $X\times_Y \widetilde{Y}\rightarrow \widetilde{Y}$ is a strict local isomorphism (resp. \'etale, smooth, flat).
\end{corollary}
We use this corollary in Section \ref{lim of topii} to show that the m-open topos is the limit of the ordinary Zariski topoi of the log blow-ups, see Proposition \ref{proposition: limit of topoi}.
\subsection{Sheaves on valuative log spaces}\label{section: intro valuative space}
Recall from \cite{kato_logarithmic_degeneration} that one can associate to any fs log scheme $X$ a certain log locally ringed space $X^{\text{val}}$ called its \emph{valuative log space}.
It turns out that the valuative log space is closely related to our m-open site:
\begin{theorema}\label{thm: sheaves on val is m}
 There is a natural equivalence of topoi
    \[
       \text{Shv}(X^{\text{val}}) = \text{Shv}\left(X_{\text{mop}}\right).
    \]   
\end{theorema}
We will use Theorem \ref{thm: sheaves on val is m} to deduce a vanishing result for sheaf cohomology on $X_{\text{mop}}$. To this end, we introduce the following definition:
\begin{definition}\label{definition of log dimension 1}
    For any fs log scheme $X$, we define its \emph{log dimension} as the Krull dimension of its valuative log space
\[
    \text{log-dim}(X)\coloneqq \dim(X^{\text{val}}).
\]
\end{definition}
Koppensteiner and Talpo define a different notion of log dimension for closed subschemes of log smooth fs log schemes over $\mathbb{C}$ \cite[Definition 2.3]{koppensteiner2019holonomic}. However, this definition differs from ours; for example, the standard log point \(X = \operatorname{Spec}(\mathbb{N} \to \mathbb{C})\) has log dimension 1 in their sense, but dimension 0 in ours.
\begin{corollary}\label{cor: cohomology vanishing}
    For any quasi-compact and quasi-separated fs log scheme $X$ and sheaf $\mathcal{F}$ of abelian groups on $X_{\text{mop}}$, we have $H^{p}(X,\mathcal{F}) = 0$ if $p>\text{log-dim}(X)$.
\end{corollary}
In order to make Corollary \ref{cor: cohomology vanishing} usable in practice, we give the following more explicit description for locally noetherian fs log schemes.
\begin{theorema}\label{Theorem on log dimensions}
If $X$ is a locally noetherian fs log scheme, then we have
    \begin{equation}\label{equation: log-dim if loc noeth}
        \text{log-dim}(X) =  \underset{x\in X}{\sup}\left[ \dim \overline{\{x\}} + \max\left(\text{rank} (\overline{\mathcal{M}}_{X,\overline{x}}^{gp})-1,0\right)\right].
    \end{equation}
\end{theorema}

\subsection{Conventions}
Throughout this paper, we ignore set-theoretical issues associated with large sites. All log schemes will be endowed with \'etale log structures that locally admit charts. An integral (resp. saturated, valuative, fs) log scheme is a log scheme that is covered by charts of integral (resp. saturated, valuative, fine and saturated) monoids. All charts in this paper are fine and saturated unless stated otherwise. We denote the category of fs log schemes by $\mathsf{fsLogSch}$ and the category of schemes by $\mathsf{Sch}$. Moreover, we will say that the log structure of a fs log scheme is \emph{Zariski} if it can be covered by open subsets that admit global charts. Because of \cite[Proposition III.1.4.1]{lecturesonlogarithmicalgebraicgeometry}, we may (and will) identify these with log structures defined on the Zariski topology. 

We largely follow the notation established in the standard reference \cite{lecturesonlogarithmicalgebraicgeometry}. For a log scheme $X$, we denote by $\underline{X}$ its underlying scheme, by $\mathcal{M}_X$ its log structure, and by $\overline{\mathcal{M}}_X = \mathcal{M}_X/\mathcal{O}^\times_X$ its characteristic sheaf. More generally, for any monoid $P$ we write $\overline{P} = P/P^\times$, where $P^\times\subset P$ is the submonoid of units. We also use $\operatorname{Spec}(P\rightarrow A)$ to denote the scheme $\operatorname{Spec}(A)$ equipped with the log structure induced by a morphism of monoids $P \to A$. In particular, we write $\mathbb{A}_P = \text{Spec}(P\rightarrow \mathbb{Z}[P])$ for short.

\subsection{Acknowledgments}
The second author is grateful to the Max Planck Institute for Mathematics in Bonn for its hospitality and financial support for part of the writing process of this paper. This paper is part of the Ph.D.\ thesis of X.H., who would like to thank his advisors, Christian Liedtke and Helge Ruddat, for their unwavering support. We thank Chikara Nakayama for very helpful correspondence and confirming the correctness of Theorem \ref{thm: shit be fucked}. We thank Simon Felten and Weikai Hu for pointing us to the pertinent paper by Binda, Park, and \O stv\ae r. We also thank Simon Felten and Helge Ruddat for useful discussions on Theorem \ref{thm: shit be fucked} and Doosung Park for catching several typos in this paper. We further thank Adam Dauser and Tianyi Feng for suggesting the title for this paper. This work began during X.H.'s visit to the University of Stavanger, whose hospitality is gratefully acknowledged. 

During the preparation of this paper, X.H. was supported by the Deutsche Forschungsgemeinschaft (DFG, German Research Foundation) -- Project No.516701553.
During part of the preparation of this paper, M.S. was supported by the European Research Council (ERC) through the starting grant `Correspondences in enumerative geometry: Hilbert schemes, K3 surfaces and modular forms' (Grant Agreement No.101041491).

\section{Abstract log modifications and m-type morphisms}
The main goal of this section is to prove basic properties of m-type morphisms and abstract log modifications. In particular, we show that abstract log modifications are proper and m-type morphisms can locally be written as compositions of log blow-ups and strict morphisms of the same type. Furthermore, exact m-type morphisms are strict.

\subsection{Log modifications} 
Several different definitions of \enquote{log modification} appear in the literature \cite{abramovich2020decomposition,abramovich2018birational,Schmitt2025logarithmic,Molcho_Wise_2022}. The one that is best suited for our purposes is inspired by the work of Doosung Park \cite{park2024inverting}:
\begin{definition}\label{defi: log mod}
A morphism $f\colon X\rightarrow Y$ of fs log schemes is called an \emph{abstract log modification} if $f$ is a universally surjective log \'etale monomorphism.
\end{definition}
For example, toric modifications and log blow-ups are abstract log modifications (see \cite[Example 2.7]{park2024inverting} and \cite[Theorem III.2.6.3(3)]{lecturesonlogarithmicalgebraicgeometry}).
We further show that any abstract log modification is automatically proper (see Proposition \ref{prop: log mods are proper}) and thus abstract log modifications coincide with the \emph{dividing covers} of \cite{Triangulated-categories-of-logarithmic-motives-over-a-field, park2019equivariant, park2024inverting}.

Throughout this paper, we call a morphism of fs log schemes universally surjective if all of its base changes in the category of integral log schemes are surjective. This is strictly stronger than the usual universal surjectivity in the category of fs log schemes (see Subsection \ref{subsection in introduction about universal surjectivity} for more details).

\subsection{Abstract log modifications and m-open morphisms}
It turns out that the m-open morphisms admit a chart criterion if the target has a Zariski log structure.
\begin{proposition}\label{chart for m-open}
Let $f\colon X\rightarrow Y$ be a morphism of fs log schemes where the log structure of $Y$ is Zariski. Then $f$ is m-open if and only if it satisfies the following chart criterion:\\
For any $x\in X$ and any open neighborhood $V\subset Y$ of $f(x)$ admitting a chart $Q\rightarrow \Gamma(V,\mathcal{M}_Y)$, there exists an open neighborhood $U\subset X$ of $x$ so that $f(U)\subset V$ and there is a morphism of charts
    \[
    \begin{tikzcd}
        Q \arrow[r] \arrow[d, "\theta"'] & \Gamma(V, \mathcal{M}_{Y}) \arrow[d] \\
        P \arrow[r] & \Gamma(U, \mathcal{M}_X)
    \end{tikzcd}
    \]
    such that:
    \begin{enumerate}
        \item the morphism $\theta$ induces an isomorphism $Q^{\mathrm{gp}} \cong P^{\mathrm{gp}}$;
        \item the induced morphism $U \to V \times_{\mathbb{A}_Q} \mathbb{A}_P$ is a strict open immersion.
    \end{enumerate}
\end{proposition}
Before we give the proof, we first establish the following technical result:
\begin{proposition}\label{prop: mono Zariski}
Let $f\colon X\rightarrow Y$ be a monomorphism between fs log schemes. If the log structure on $Y$ is Zariski, then the log structure of $X$ is also Zariski.
\end{proposition}
\begin{proof}
Let $\underline{X}^{\text{\'et}}$ and $\underline{X}^{\text{Zar}}$ be the small \'etale and Zariski sites of the underlying scheme of $X$ and $\epsilon_*\colon \text{LogStr}(\underline{X}^{\text{\'et}}) \rightarrow \text{LogStr}(\underline{X}^{\text{Zar}})$ and $\epsilon^*\colon \text{LogStr}(\underline{X}^{\text{Zar}}) \rightarrow \text{LogStr}(\underline{X}^{\text{\'et}})$ be the associated pushforwards and pullbacks of log structures. Let $X^{\text{Zar}}$ denote the scheme $\underline{X}$ equipped with the log structure $\mathcal{M}_X^{\text{Zar}} \coloneqq \epsilon^*\epsilon_*\mathcal{M}_X$. There is a natural morphism $\phi\colon X\rightarrow X^{\text{Zar}}$ induced by the counit $\mathcal{M}_X^{\text{Zar}}\rightarrow \mathcal{M}_X$. Since $\epsilon^*$ is fully faithful, one can see that $\mathcal{M}_X^{\text{Zar}}$ is terminal among Zariski log structures mapping to $\mathcal{M}_X$. As a result, $X^{Zar}$ is initial among Zariski log schemes under $X$. Since $Y$ is Zariski, it follows that $f$ factors over $\phi$; hence $\phi$ is also a monomorphism\footnote{Note that $\mathcal{M}_X^{\text{Zar}}$ may not be quasi-coherent, so $\phi$ only exists in a much bigger category of log schemes. By saying that $\phi$ is a monomorphism, we mean only that for any fs log scheme $T$ and morphisms $T \substack{f\\ \rightrightarrows \\ g} X$ that equalize $\phi$, we must have $f=g$.}.

In order to show that $\phi$ is an isomorphism, it suffices by \cite[Proposition 2.8]{Toric-singularies} to show that for any $x\in X$, there is no non-trivial automorphism of $\overline{\mathcal{M}}_{X,\overline{x}}$ that fixes $\overline{\mathcal{M}^{\text{Zar}}}_{X,\overline{x}}\rightarrow \overline{\mathcal{M}}_{X,\overline{x}}$. Since $\phi$ is a monomorphism, the base change $\overline{x}^{\text{\'et}}\rightarrow \overline{x}^{\text{Zar}}$ of $\phi$ along the geometric point $\overline{x}^{\text{Zar}}=\text{Spec}(k) \rightarrow X^{\text{Zar}}$ equipped with the induced log structure is also a monomorphism. By \cite[Remark III.1.5.3]{lecturesonlogarithmicalgebraicgeometry}, the log structure on $\overline{x}^{\text{\'et}}$ splits into $\mathcal{M}_{\overline{x}^{\text{\'et}}} = k^\times \oplus \overline{\mathcal{M}}_{X,\overline{x}}$. Hence any automorphism of $\overline{\mathcal{M}}_{X,\overline{x}}$ that fixes $\overline{\mathcal{M}^{\text{Zar}}}_{X,\overline{x}}$ can be lifted to an automorphism of $\mathcal{M}_{\overline{x}^{\text{\'et}}}$ that fixes $\mathcal{M}_{\overline{x}^{\text{Zar}}}$. But since $\overline{x}^{\text{\'et}}\rightarrow \overline{x}^{\text{Zar}}$ is a monomorphism, this automorphism must have been the identity, which concludes the proof.
\end{proof}
\begin{remark}\label{remark: m-open means Zariski}
    In particular, this shows that if $f\colon X\rightarrow Y$ is m-open and $Y$ has a Zariski log structure, then $X$ has a Zariski log structure as well.
\end{remark}
\begin{proof}[Proof of Proposition \ref{chart for m-open}]
If $f$ is m-open, it is enough to check the chart criterion on an open cover of $X$, hence we may assume that $f$ is a log \'etale monomorphism. By Proposition \ref{prop: mono Zariski}, $X$ will then also have a Zariski log structure and \cite[Lemma A.11.3]{Triangulated-categories-of-logarithmic-motives-over-a-field} gives us the desired charts. Conversely, the chart criterion implies that $f$ admits, Zariski-locally, a factorization of the shape $X\hookrightarrow Y\times_{\mathbb{A}_Q}\mathbb{A}_P\rightarrow Y$ with $P^{gp} = Q^{gp}$, where the first map is a strict open immersion. Such a composite is log \'etale and to show that it is a monomorphism, one can use the construction of the fiber product in the category of fs log schemes or alternatively Lemma \ref{Lemma for special type of morphisms of monoids}.
\end{proof}
\begin{lemma}\label{Lemma for special type of morphisms of monoids}
Let $Q\rightarrow R$ be a pre-log ring and $Q\rightarrow P$ a morphism of fine, saturated monoids which induces an isomorphism $Q^{gp}\cong P^{gp}$. The resulting morphism $f\colon\text{Spec}(P\rightarrow R\otimes_{\mathbb{Z}[Q]}\mathbb{Z}[P])\rightarrow \text{Spec}(Q\rightarrow R)$ can be factored into the composition of an open immersion and a log blow-up.
\end{lemma}
\begin{proof}
Since $Q\rightarrow P$ induces an isomorphism $Q^{gp}\cong P^{gp}$ and $Q$ is integral, we know that $Q\subset P$. Since $P$ is finitely generated and $P\subset Q^{gp}$, we may choose generators $\frac{f_1}{s},...,\frac{f_n}{s}$ of $P$, where $f_i,s\in Q$. Let $I$ be the ideal generated by $f_1,f_2,...,f_n, s$. As a result, we get $P=\bigcup_{n\geq 0} \frac{I^{n}}{s^n}$ and by the construction of the log blow-up (see \cite[Section 2.1]{NAKAYAMALOGETALECohomology2} for details), the morphism $f$ factors as desired:
\[\text{Spec}(P\rightarrow R\otimes_{\mathbb{Z}[Q]}\mathbb{Z}[P])=D_{+}(s)\hookrightarrow\text{Bl}_I(\text{Spec}(Q\rightarrow R))\rightarrow\text{Spec}(Q\rightarrow R).\]
\end{proof}
\begin{remark}\label{m-open local composition}
    Using Lemma \ref{Lemma for special type of morphisms of monoids}, one can rephrase Proposition \ref{chart for m-open} in the following way: If $Y$ is Zariski, then $f$ is m-open if and only if Zariski locally in $X$ and $Y$, $f$ can be written as the composition of a strict open immersion followed by a log blow-up. In particular, this implies that the morphism $f^{\text{val}}\colon X^{\text{val}}\rightarrow Y^{\text{val}}$ of Remark \ref{you can always blow up} is a local isomorphism if $f$ is m-open - even if $Y$ is not Zariski.
\end{remark}
Definitions of log modifications usually require properness as an extra assumption. We show that this is automatic from our definition.
\begin{proposition}\label{prop: log mods are proper}
Any abstract log modification is proper.
\end{proposition}
\begin{proof}
Since both notions are strict \'etale local in the target (for properness, see \cite[02L1]{stacks-project}), it suffices to consider the case where the target $Y$ of the abstract log modification $f\colon X\rightarrow Y$ is affine and admits a global chart $P\rightarrow \Gamma(Y,\mathcal{M}_Y)$. By Proposition \ref{chart for m-open}, we can find a Zariski cover $\{U_i\rightarrow X\}_{i\in I}$ so that each $U_i\rightarrow Y$ factors, by Lemma \ref{Lemma for special type of morphisms of monoids}, as a composition $U_i\rightarrow \widetilde{Y}_i \rightarrow Y$ of an open immersion and the log blow-up at an ideal $I_i\subset P$. Since $f$ is universally surjective, the $U_i\rightarrow Y$ form an m-open cover of $Y$. By Proposition \ref{prop: quasi-compact is quasi-compact} and quasi-compactness of $Y$, there is a finite subset $J\subset I$ so that $\{U_i\rightarrow Y\}_{i\in J}$ is again a cover.

Now let $\widetilde{Y}\rightarrow Y$ be the log blow-up at $I=\prod_{j\in J} I_j$ so that $\widetilde{Y}\rightarrow Y$ factors through $\widetilde{Y}_i\rightarrow Y$ for all $i\in J$. Since log blow-ups are monomorphisms, it follows that each base change $\widetilde{U}_i\coloneqq U_i\times_Y \widetilde{Y}\rightarrow\widetilde{Y}$ is an open immersion. Hence, $\widetilde{U}\coloneqq \bigcup_{j\in J} \widetilde{U}_j\subset \widetilde{X}\coloneqq X\times_Y\widetilde{Y}$ is an open subset such that $\widetilde{U}\rightarrow\widetilde{X}\rightarrow\widetilde{Y}$ is a surjective local isomorphism that is also a monomorphism and hence an isomorphism. As a result, $\widetilde{X}\rightarrow \widetilde{Y}$ is a monomorphism with a section, which forces it to be an isomorphism as well. Thus, the composite $\widetilde{X}\rightarrow X\rightarrow Y$ is a log blow-up, which by \cite[09MQ (4)]{stacks-project} means that $X\rightarrow Y$ is separated and by \cite[03GN]{stacks-project} universally closed. Since $X\rightarrow Y$ is log \'etale, it is locally of finite type. Lastly, $X$ is also quasi-compact since it is dominated by $\widetilde{X}$, which is quasi-compact as it is a log blow-up of $Y$ and $Y$ was quasi-compact. Therefore, $X\rightarrow Y$ is proper as desired.
\end{proof}
\begin{remark}
\begin{enumerate}
    \item Since we appealed to Proposition \ref{prop: quasi-compact is quasi-compact}, this proof does not work unless one uses our notion of universal surjectivity. Indeed, one can use the covering in Proposition \ref{proposition: the counterexample} to construct examples of (weakly) universally surjective log \'etale monomorphisms that are not proper.
    \item The above proof shows that any abstract log modification in our sense is also a log modification in Kato's sense \cite{kato2021exactnessintegralitylogmodifications}.
\end{enumerate}
\end{remark}
\subsection{Basic properties}
Here we collect a few of the basic properties of m-type morphisms.
\begin{proposition}\label{basic properties of m-type morphisms}
\begin{enumerate}
\item\label{m-open is m-etale} Every m-open morphism is m-\'etale.
\item\label{m-is-log} Any m-\'etale (resp. m-smooth, m-flat) morphism is also log \'etale (resp. log smooth, log flat).
\item\label{strict-is-m} Local isomorphisms (resp. strict \'etale, smooth, flat morphisms) are m-open (resp. m-\'etale, m-smooth, m-flat).
\item\label{m-is-s-u-com} The conditions of being m-open, m-\'etale, m-smooth and m-flat are stable under composition and base change.
\item\label{locality} A morphism $f\colon X\rightarrow Y$ is m-\'etale (resp. m-smooth, m-flat) if and only if it can be written, strict \'etale locally in $X$ and $Y$, as a composition of a strict \'etale (resp. strict smooth, strict flat) morphism followed by a log blow-up. In particular, being m-\'etale (resp. m-smooth, m-flat) is strict \'etale local in source and target.
\end{enumerate}    
\end{proposition}
\begin{proof}
Properties \eqref{m-is-log} and \eqref{strict-is-m} and stability under composition in \eqref{m-is-s-u-com} follow directly from the definitions. Stability under base change is immediate for m-open morphisms and follows from \eqref{locality} for m-\'etale, m-smooth and m-flat morphisms. It furthermore follows from \eqref{locality} that any m-open morphism is m-\'etale since one can replace the target by an \'etale cover with a Zariski log structure and then apply Remark \ref{m-open local composition}. Hence, it only remains to show \eqref{locality}. 

For that, we first note that any m-\'etale (resp. m-smooth, m-flat) morphism is of the described shape since one can cover $Y$ by charts $U\rightarrow Y$ for which we can find a compatible chart $V\rightarrow X$ so that $U\rightarrow V$ factors by Lemma \ref{Lemma for special type of morphisms of monoids} in the desired way. For the converse, first note that log blow-ups are m-\'etale as can be seen from their explicit construction. Together with \eqref{strict-is-m}, this shows that any composition as in \eqref{locality} is m-\'etale (resp. m-smooth, m-flat). It therefore suffices to show that the properties of being m-\'etale (resp. m-smooth, m-flat) are \'etale local properties. 

Indeed, \'etale locality in the source follows directly from the definition. For \'etale locality in the target we let $Y'\twoheadrightarrow Y$ be an \'etale cover so that $ X' = X\times_Y Y'\rightarrow Y'$ is m-\'etale (resp. m-smooth, m-flat). Now let $x\in X$, and let $V\rightarrow Y$ be an \'etale neighborhood of $f(x)$ with a chart $Q\rightarrow \Gamma(V,\mathcal{M}_Y)$ and let $x'\in X'$ be some preimage of $x$. By assumption, there is an \'etale neighborhood $U\rightarrow X'$ of $x'$ with a chart $P\rightarrow \Gamma(U,\mathcal{M}_{X'})$ so that $U\rightarrow X'\rightarrow Y'$ factors through $V' = V\times_Y Y'\rightarrow Y'$. We further get that there is a morphism of charts $Q\rightarrow P$ so that $Q^{gp}=P^{gp}$ and the resulting $U\rightarrow V'' \coloneqq V'\times_{\mathbb{A}_Q} \mathbb{A}_P$ is strict \'etale (resp. smooth, flat). However, $V''\rightarrow V''' \coloneqq V\times_{\mathbb{A}_Q} \mathbb{A}_P$ is strict \'etale as it is a base change of $Y'\rightarrow Y$ and so the resulting composite $U\rightarrow V''\rightarrow V'''$ is again strict \'etale (resp. smooth, flat) and hence the chart $U \rightarrow X' \rightarrow X$ is the desired one.
\end{proof}
\begin{example}
    The following example illustrates why m-open morphisms are different than the other m-type morphisms. Let $X=\text{Spec}(\mathbb{N}^2\xrightarrow{0}\mathbb{C})$ and $Y = \text{Spec}(\mathbb{R})$ equipped with the log structure descended from $X$ by imposing that complex conjugation swaps the two factors of $\mathbb{N}^2$. Hence all morphisms in the cartesian square
    \begin{center}
    \begin{tikzcd}
        X\sqcup X = X\times_Y X \rar\dar &X\dar\\
        X\rar & Y
    \end{tikzcd}
    \end{center}
    are strict \'etale covers and the upper horizontal morphism is a local isomorphism and thus m-open. However, $X\rightarrow Y$ is not a monomorphism and thus not m-open. 
    
    Hence being m-open is not \'etale local in the target (unlike the other m-type morphisms). Moreover, this demonstrates why we did not define m-open morphisms via a chart criterion, because the natural analog of Definition \ref{Definition of m-type morphism} also holds for $X\rightarrow Y$. Hence, this does not capture the notion of being a combination of abstract log modifications and open immersions.
\end{example}
We now show that m-type morphisms lie at the opposite extreme from Kummer type morphisms.
\begin{proposition}\label{m-type of Kummer is strict}
Let $f\colon X\rightarrow Y$ be an m-open (resp. m-\'etale, m-smooth, m-flat) morphism. If $f$ is also exact, then $f$ is a strict local isomorphism (resp. strict \'etale, strict smooth, strict flat). In particular, an m-open (resp. m-\'etale, m-smooth, m-flat) morphism of Kummer type is a strict local isomorphism (resp. strict \'etale, strict smooth, strict flat).
\end{proposition}
\begin{proof} 
For m-open morphisms, this follows from the fact that an exact log \'etale monomorphism has to be strict open (see \cite[Proposition 2.8]{park2024inverting}). 

For the other cases, we recall from Proposition \ref{basic properties of m-type morphisms}\eqref{m-is-log} that a strict m-\'etale (resp. m-smooth, m-flat) morphism is also a strict log \'etale (resp. log smooth, log flat) and by \cite[Proposition IV.3.1.6]{lecturesonlogarithmicalgebraicgeometry} and \cite[Proposition IV.4.1.2(2)]{lecturesonlogarithmicalgebraicgeometry}, this forces it to be strict \'etale (resp. smooth, flat). 

Now assume that $f$ is exact. By the previous discussion, it suffices to show that $f$ is strict; that is, for any $x\in X$, we have $\overline{\mathcal{M}}_{X,\overline{x}} = \overline{\mathcal{M}}_{Y,\overline{f(x)}}$. By Proposition \ref{basic properties of m-type morphisms}\eqref{locality}, we have the following diagram:
\begin{equation}
\begin{tikzcd}
U \arrow[r] \arrow[d, "f'"] & X \arrow[d, "f"] \\
V \arrow[r]  & Y.              
\end{tikzcd}    
\end{equation}
Here, $U$ and $V$ are \'etale neighborhoods of $x$ and $f(x)$ and the horizontal morphisms are strict \'etale. Moreover, $U\rightarrow V$ factors as $U\xrightarrow{g}\widetilde{V}\xrightarrow{h}V$, where $g$ is a strict \'etale (resp. smooth, flat) morphism and $h$ is a log blow-up. Since $U\rightarrow V$ is exact at $x$ and $g$ is strict, it follows that $h$ is exact at $g(x)$. By \cite[Proposition I.4.2.1(5)]{lecturesonlogarithmicalgebraicgeometry}, it follows that $\overline{\mathcal{M}}_{V,\overline{g(x)}}\rightarrow \overline{\mathcal{M}}_{\widetilde{V},\overline{f(x)}}$ is injective. Since $h$ is a log blow-up, we moreover have that $\overline{\mathcal{M}}_{V,\overline{g(x)}}^{gp}\rightarrow \overline{\mathcal{M}}_{\widetilde{V},\overline{f(x)}}^{gp}$ is surjective, which by exactness also implies that $\overline{\mathcal{M}}_{V,\overline{g(x)}}\rightarrow \overline{\mathcal{M}}_{\widetilde{V},\overline{f(x)}}$ is surjective. This implies the desired strictness.

A morphism of Kummer type is also exact (see \cite[Chapter I, Corollary 4.3.10]{lecturesonlogarithmicalgebraicgeometry}), hence, an m-type morphism of Kummer type is a strict morphism of the same type.
\end{proof}
\begin{remark}
    In particular, it follows that an abstract log modification is an isomorphism if and only if it is exact.
\end{remark}
\section{The lifting criterion for m-\'etale and m-smooth morphisms}
The goal of this section is to prove Theorem \ref{main theorem A}. We begin with some preliminary results on formally m-\'etale (m-smooth) morphisms.
\subsection{Strict \'etale morphisms and log blow-ups are formally m-\'etale}
As for the classical notions of formal \'etaleness and formal smoothness, we have:
\begin{lemma}\label{Properties for m-etale and smooth morphisms}
\begin{enumerate}
    \item Being formally m-\'etale or m-smooth is stable under base change and composition.
    \item\label{item: useful composition property} Let $f\colon X\rightarrow Y$ and $g\colon Y\rightarrow Z$ be morphisms of fs log schemes. If the composition $g\circ f\colon X\rightarrow Z$ is formally m-\'etale (resp. {m-smooth}) and $g$ is formally m-\'etale, then $f$ is formally m-\'etale (resp. m-smooth).
\end{enumerate}
 \end{lemma}
\begin{proof}
The proofs of both statements are identical to the standard arguments for formally étale and formally smooth morphisms in $\mathsf{Sch}$.
\end{proof}
\begin{lemma}\label{strict etale is m-etale} A strict morphism is formally m-\'etale (resp. m-smooth) if and only if the underlying morphism of schemes is formally \'etale (resp. smooth).
\end{lemma}
\begin{proof}
For any strict morphism $f\colon X\rightarrow Y$ of log schemes, the following diagram 
\begin{equation}
\begin{tikzcd}
X \arrow[r, "p_X"] \arrow[d, "f"] \ar[dr, phantom, "\lrcorner", very near start]       & X^{\text{triv}} \arrow[d, "f^{\text{triv}}"] \\
Y \arrow[r, "p_Y"] & Y^{\text{triv}}     
\end{tikzcd}
\end{equation}
is cartesian by \cite[III, Proposition 1.2.5]{lecturesonlogarithmicalgebraicgeometry}, where $X^{\text{triv}}$ and $Y^{\text{triv}}$ are $\underline{X}$ and $\underline{Y}$ equipped with trivial log structures and $p_X$ and $p_Y$ are the canonical morphisms. Thus, a lift of $T\rightarrow Y$ to $X$ is the same as a lifting of $T^{\text{triv}}\rightarrow Y^{\text{triv}}$ to $X^{\text{triv}}$.
\end{proof}
\begin{proposition}\label{log blow-up is m-etale}
Let $f\colon X\rightarrow Y$ be a log blow-up. For any commutative diagram of the form below
\begin{center}
\begin{tikzcd}
T' \arrow[r, "l"] \arrow[d, "i"]                 & X \arrow[d, "f"] \\
T \arrow[r, "g"] \arrow[ru, "{\exists!\,h}", dashed] & Y               
\end{tikzcd}
\end{center}
where $i$ is exact and surjective, there is a unique lift $h$ as indicated in the above diagram. In particular, $f$ is formally m-\'etale.
\end{proposition}
\begin{proof}
Let $\mathcal{I}\subset \mathcal{M}_Y$ be the coherent ideal sheaf such that $X=\text{Bl}_{\mathcal{I}}(Y)$. By the universal property of the log blow-up, the lift $h$ exists if and only if $g^*\mathcal{I}$ is \'etale locally principal and $h$ is then also unique. By \cite[III, Proposition 2.6.1]{lecturesonlogarithmicalgebraicgeometry}, it suffices to show that $g^*\mathcal{I}_{\overline{t}}$ is generated by one element for every $t\in T$. As $i$ is surjective, there exists $t'\in T'$ with $i(t')=t$. Since $i^*g^*\mathcal{I}_{T',\overline{t'}}$ is principal and generated by $i_{log}(g^*\mathcal{I}_{\overline{t}})$, we can find a generator of the shape $i_{log}(a)$ for $a\in \mathcal{M}_{T,\overline{t}}$. We now claim that $a$ also generates $g^*\mathcal{I}_{\overline{t}}$ itself. Indeed, for any $b\in g^*\mathcal{I}_{\overline{t}}$ we have $i_{log}(b)\in i_{log}(a)\cdot \mathcal{M}_{T',\overline{t'}}$ or equivalently $i_{log}(b)/ i_{log}(a)\in \mathcal{M}_{T',\overline{t'}}$. By exactness of $i$ this implies that $b/a\in \mathcal{M}_{T,t}$ as desired.
\end{proof}
\subsection{Proof of Theorem \ref{main theorem A} and consequences}
To prove Theorem \ref{main theorem A}, we introduce the following technical lemma:
\begin{lemma}\label{Nifty Chart}
Let $f\colon X\rightarrow Y$ be a morphism in $\mathsf{fsLogSch}$ and $\beta\colon Q\rightarrow \Gamma(Y,\mathcal{M}_Y)$ be a chart. For any point $x\in X$, there is an \'etale neighborhood $U$ of $x$ and a morphism of charts
\begin{equation}
\begin{tikzcd}
Q \arrow[r, "\beta"] \arrow[d, "\theta"] & {\Gamma(Y,\mathcal{M}_{Y})} \arrow[d, "f^{log}"] \\
P \arrow[r, "\alpha"]                    & {\Gamma(U,\mathcal{M}_X)}                      \end{tikzcd}
\end{equation}
with the following properties:
\begin{enumerate}
    \item $\alpha$ is exact at $x$ or equivalently,  $\overline{P}\xrightarrow{\overline{\alpha}_{\overline{x}}}\overline{\mathcal{M}}_{X,\overline{x}}$ is an isomorphism.
    \item $\theta$ is injective.
\end{enumerate}
\end{lemma}
\begin{proof}
Let $U\rightarrow X$ be any \'etale neighborhood of $x$ together with a chart for $f$ given by $\alpha'\colon P'\rightarrow \Gamma(U,\mathcal{M}_X)$ and $Q\xrightarrow{\theta'} P'$. Note that $P'\oplus Q^{gp}\xrightarrow{\pi_{P'}} P' \xrightarrow{\alpha'} \Gamma(U,\mathcal{M}_X)$ is again a chart at $x$ and $Q\xrightarrow{(\theta',\text{incl})}P'\oplus Q^{gp}$ is injective, hence we may assume that $\theta'$ is already injective. 

It remains to ensure that the chart is exact at $x$. For this, note that $S\coloneqq (\alpha'_{\overline{x}})^{-1}\mathcal{O}^\times_{X,\overline{x}}\subset P'$ is a finitely generated submonoid as it is generated by a subset of the generators of $P'$. By shrinking $U$ we may therefore assume that $S = (\alpha')^{-1}\mathcal{O}^\times_X(U)$ and hence we get an induced morphism $\alpha\colon S^{-1}P'\rightarrow \Gamma(U,\mathcal{M}_X)$ that is again a chart so that $S^{-1}P'\rightarrow \mathcal{M}_{X,\overline{x}}$ is exact. Thus, $P=S^{-1}P'$ and $\theta: Q\xhookrightarrow{\theta'} P'\hookrightarrow P$ are as desired.
\end{proof}

\begin{proof}[Proof of Theorem \ref{main theorem A}]
By Proposition \ref{basic properties of m-type morphisms}\eqref{locality}, m-\'etale (resp. m-smooth) are formally m-\'etale (resp. m-smooth) as strict \'etale (resp. smooth) morphisms and log blow-ups are formally m-\'etale (resp. m-smooth).

For the converse, let $f\colon X\rightarrow Y$ be locally of finite presentation and formally m-\'etale (resp. m-smooth) and suppose that $Y$ admits a global chart $Q\rightarrow \Gamma(Y,\mathcal{M}_Y)$.
By Lemma \ref{Nifty Chart}, for every $x\in X$ there is  an \'etale neighborhood $U\rightarrow X$ of $x$ and a chart of $f$ given by $Q\hookrightarrow P'$ and $P'\rightarrow \Gamma(U,\mathcal{M}_X)$ that is exact at $x$.

Let $P\coloneqq P'\cap Q^{gp}$, which by \cite[Theorem I.2.1.17(2)]{lecturesonlogarithmicalgebraicgeometry} is fine and saturated. Since $P'$ is an exact chart at $\overline{x}$ and $P\hookrightarrow P'$ is an exact morphism of monoids, it follows from \cite[Proposition III.2.2.2(1)]{lecturesonlogarithmicalgebraicgeometry} that the induced morphism of log structures $P^{a}\rightarrow P'^{a} = \mathcal{M}_X|_{U}$ is exact at $x$ and $P$ is an exact chart of $P^{a}$ at $x$. In particular, we have the following diagram
\begin{equation}\label{the first commutative diagram for Kato-type criterion}
\begin{tikzcd}
\overline{P} \arrow[d, "\overline{\alpha}_{\overline{x}}"',"\cong"] \arrow[r, "\phi", hook]         & \overline{P'} \arrow[d, "\overline{\alpha'}_{\overline{x}}"',"\cong"] \\
\overline{P^{a}}_{\overline{x}} \arrow[r, "\phi_{\overline{x}}", hook] & \overline{P'^{a}}_{\overline{x}}.              
\end{tikzcd}   
\end{equation}
The morphisms $\overline{\alpha}_{\overline{x}}$ and $\overline{\alpha'}_{\overline{x}}$ are isomorphisms by the assumption of exactness. Since both morphisms $\phi$ and $\phi_{\overline{x}}$ are exact and both monoids $\overline{P}$ and $\overline{P'}$ are sharp, it follows from \cite[Proposition I.2.1.16(2)]{Log-etale-cohomology-I} that $\phi$ and $\phi_{\overline{x}}$ are injective.

By pulling back those log structures to $\text{Spec}(k(x))$, we get the following diagram 
\begin{equation}\label{the second diagram in the kato-type chart criterion}
\begin{tikzcd}
\text{Spec}(P'\rightarrow k(x)) \arrow[d, "i"] \arrow[r, hook]                & X \arrow[d, "f"] \\
\text{Spec}(P\rightarrow k(x)) \arrow[r, "\tilde{f}"] \arrow[ru, "g", dashed] & Y              
\end{tikzcd}
\end{equation}
where $\tilde{f}(x)=f(x)$.  
Since $i$ is exact (being a pullback of $(U,P'^{a})\rightarrow (U,P^{a})$, which is exact at $x$) and also a (trivial) nilpotent thickening, we get a lift $g$ as shown in the diagram. By commutativity of the upper triangle, this gives a section of the injection $\overline{P^{a}}_{\overline{x}}\hookrightarrow \overline{P'^{a}}_{\overline{x}}$ and hence $\overline{P^{a}}\cong \overline{P'^{a}}$. By \eqref{the first commutative diagram for Kato-type criterion}, we get $\overline{P}\cong \overline{P'}$, which implies that $P$ is also a chart of $\mathcal{M}_X$ in some possibly smaller \'etale neighborhood of $x$, which we might as well assume to be $U$. Hence, we found a chart
\begin{equation}
\begin{tikzcd}
Q \arrow[r] \arrow[d] & {\Gamma(Y,\mathcal{M}_Y)} \arrow[d] \\
P \arrow[r]           & {\Gamma(U,\mathcal{M}_X)}          
\end{tikzcd}
\end{equation}
for $f$ which satisfies the desired conditions $Q\hookrightarrow P$ and $Q^{gp}=P^{gp}$. 

Finally, we consider the following commutative diagram:
\begin{equation}
\begin{tikzcd}
    U \arrow[rd, "f_4"'] \arrow[r, "f_1"] \arrow[rr, "f_2", bend left] & {Y\times_{\mathbb{A}_Q}\mathbb{A}_P} \arrow[d, "f_5"] \arrow[r, "f_3"] \ar[dr, phantom, "\lrcorner", very near start] & Y \arrow[d, "f_7"]\\
    & {\mathbb{A}_P} \arrow[r, "f_6"] & \mathbb{A}_Q.
\end{tikzcd}
\end{equation}
By Lemma \ref{Lemma for special type of morphisms of monoids}, $f_6$ is m-\'etale, hence, $f_3$ is also m-\'etale. Since $f_2$ is formally m-\'etale (resp. m-smooth), by Lemma \ref{Properties for m-etale and smooth morphisms}\eqref{item: useful composition property}, $f_1$ is formally m-\'etale (resp. m-smooth). However, $f_1$ is strict as $f_4$ is strict and $f_5$, being the base change of $f_7$, is also strict. By Lemma \ref{strict etale is m-etale}, this implies that $f_1$ is formally \'etale (resp. smooth) and since $f$ and $f_3$ are locally of finite presentation, it follows from \cite[02FV]{stacks-project} that $f_1$ is locally of finite presentation. This concludes the proof.
\end{proof}
\begin{corollary}\label{corollary: property that makes small site well-behaved}
    Let $f\colon X\rightarrow Y$ and $g\colon Y\rightarrow Z$ be morphisms of fs log schemes. If $g$ and $g\circ f$ are m-\'etale (resp. m-open), then $f$ must be m-\'etale (resp. m-open).
\end{corollary}
\begin{proof}
    For m-\'etale morphisms, this follows directly from Theorem \ref{main theorem A}, Lemma \ref{Properties for m-etale and smooth morphisms}\eqref{item: useful composition property} and \cite[02FV]{stacks-project}. Now we will deal with the case of m-open morphisms. For $x\in X$, choose an open neighborhood $U\subset X$  of $x$ so that $g\circ f|_U$ is a log \'etale monomorphism. Since $g$ is log \'etale, it follows that $f|_U$ is also a log \'etale monomorphism. This shows that $f$ is also m-open.
\end{proof}
Moreover, for étale, smooth, and flat morphisms in the category of schemes, the following descent property holds:
For any sequence of morphisms
\[
X \xrightarrow{f} Y \xrightarrow{g} Z,
\]
if $g \circ f$ is étale (resp. smooth), $f$ is smooth and surjective, and $g$ is locally of finite presentation, then $g$ is necessarily étale (resp. smooth). It is natural to ask whether the same property also holds for m-type morphisms. However, this fails, as shown by the following example due to Chikara Nakayama:
\begin{example}\cite[Example IV.4.3.4]{lecturesonlogarithmicalgebraicgeometry}\label{the example pointing the mistake on abstract log modification}
Consider the following diagram:
\begin{center}
\begin{tikzcd}
X \arrow[d, "h"] \arrow[r, "f"] \ar[dr, phantom, "\lrcorner", shift={(-1.5,0.3)}]  & {Y\coloneqq\text{Spec}(\mathbb{N}^2\rightarrow\mathbb{C}[x,y]/(x^2,y^2,xy))} \arrow[d, "g", hook'] \\
\widetilde{Z} \arrow[d] \arrow[r,"\pi'"] \ar[dr, phantom, "\lrcorner", shift={(-2,0.3)}]           & {Z\coloneqq\text{Spec}(\mathbb{N}^2\rightarrow\mathbb{C}[x,y]/(x^2,y^2))} \arrow[d, "i", hook']    \\
\text{Bl}_0(\mathbb{A}^2) \arrow[r,"\pi"] & \mathbb{A}^2.                              
\end{tikzcd}  
\end{center}
The spaces $\mathbb{A}^2$ and $\mathrm{Bl}_0(\mathbb{A}^2)$ carry standard toric log structures, and $\pi$ is a toric blow-up (hence a log blow-up). Since log blow-ups are stable under base change, it follows that $f$ and $\pi'$ are also log blow-ups and thus universally surjective by Proposition \ref{prop: universally surjective examples}\eqref{item: log blow is univ surj}. Either from the universal property of the log blow-up or through a direct computation, one can verify that $h$ is an isomorphism and hence $g\circ f$ is a log blow-up. However, $g$ is not even log flat since it is strict and its underlying morphism of schemes is not flat.
\end{example}   
\section{M-type topologies}
First, we establish the following:
\begin{lemma}\label{lemma: m-type tops are tops}
    The m-type covers of Definition \ref{definition of m-type coverings} form Grothendieck topologies.
\end{lemma}
\begin{proof}
    For the m-open, m-\'etale, m-smooth and m-fppf topologies, this follows from Proposition \ref{basic properties of m-type morphisms} \eqref{m-is-s-u-com}. For the m-fpqc topology, it remains to check stability of covers under base change. Indeed, let $f\colon X\rightarrow Y$ be an m-fpqc cover and $g\colon Y'\rightarrow Y$ and a morphism of fs log schemes. To show that $f'\colon X'=X\times_Y Y'\rightarrow Y'$ is m-fpqc, it suffices to consider quasi-compact open subsets $U\subset Y'$ so that $g(U)\subset W$ for some affine open $W\subset Y$. Since $f$ is m-fpqc, there is a quasi-compact open $V\subset X$ so that $W$ is the universal image of $V$. Since the scheme-theoretic base change $\underline{U}\times_{\underline{W}}\underline{V}$ is quasi-compact and $\underline{U\times_W V}\rightarrow \underline{U}\times_{\underline{W}}\underline{V}$ is affine, it follows that $U\times_W V\subset X'$ is quasi-compact and universally surjects onto $U$.
\end{proof}
Next, we first show that the m-type sites are functorial. The following fact will be useful:
\begin{proposition}\label{m-topologies has finite limits}
All big and small m-type sites over $S$ in Definition \ref{definition of m-type sites} admit all finite limits, which agree with the limits in $\mathsf{fsLogSch}/S$.
\end{proposition}
\begin{proof}
Since all big and small sites have a terminal object, it is enough to show the existence of fiber products. For the big sites, the claim is immediate. For the small m-open (resp. m-\'etale) site, one uses Corollary \ref{corollary: property that makes small site well-behaved}, which implies that the morphisms in that site are also m-open (resp. m-\'etale). Hence, the usual fiber products are indeed objects in this site.
\end{proof}
\begin{remark}\label{remark on functoriality}
Let $\tau \in \{ \text{m-open}, \text{m-\'etale}, \text{m-smooth},\text{m-fppf},\text{m-fpqc}\}$ and $\tau'\in  \{ \text{m-open}, \\\text{m-\'etale}\}$. 
For any morphism $f\colon X\rightarrow Y$ of fs log schemes, the pullback functor $f^*\colon \mathsf{fsLogSch}/Y\rightarrow \mathsf{fsLogSch}/X$ 
preserves finite limits. By Proposition~\ref{m-topologies has finite limits}, the sites $(\mathsf{fsLogSch}/Y)_\tau$ and $(\mathsf{fsLogSch}/X)_\tau$ 
(resp. $Y_{\tau'}$ and $X_{\tau'}$) admit all finite limits and the resulting pullback functors $f^*\colon (\mathsf{fsLogSch}/Y)_\tau\rightarrow (\mathsf{fsLogSch}/X)_\tau$ (resp. $Y_{\tau'}\rightarrow X_{\tau'}$) preserve them. Consequently, $f^*$ induces morphisms of sites for all our big and small sites, and hence morphisms of their topoi
(see \cite[Theorem VII.10.2]{SheavesInGeometryAndLogic}).
\end{remark}
The following proposition will be important for the proof of Theorem \ref{thm:sheaf_conditions} and contains a corrected version of \cite[Lemma 3.14 (1)]{NAKAYAMALOGETALECohomology2}, from which the rest of \cite[Lemma 3.14]{NAKAYAMALOGETALECohomology2} follows immediately.
\begin{proposition}\label{prop: quasi-compact is quasi-compact}
An object in any of the m-type sites or the log \'etale site is quasi-compact (resp. quasi-separated) if and only if it is quasi-compact (resp. quasi-separated) on the level of underlying topological spaces.
\end{proposition}
\begin{proof}
    Note that integralization and saturation are quasi-compact morphisms; hence it suffices to show the quasi-compactness assertion.
    Since the above topologies contain strict open covers, an fs log scheme is quasi-compact if it is quasi-compact in one of these sites. For the converse, let $\{U_i\rightarrow U\}_{i\in I}$ be a cover, where the underlying scheme $\underline{U}$ is quasi-compact. In the m-fpqc case, it follows from the definition that we can find a finite subcover. Otherwise, the cover is log flat and locally finitely presented, which by \cite[Theorem 3.2.2]{kato_logarithmic_degeneration} implies that all morphisms $U_i^{\text{val}}\rightarrow U^{\text{val}}$ have open image. Since $U^{\text{val}}$ is quasi-compact by Theorem \ref{thm: valuative space}\eqref{item: properties of valuative space}, there is a finite $I'\subset I$ so that $\coprod_{i'\in I'} U_{i'}^{\text{val}} \rightarrow U^{\text{val}}$ is still surjective. By Lemma \ref{lemma: universal surjectivity def}, the family $\{U_{i'}\rightarrow U\}_{i'\in I'}$ is our desired subcover.
\end{proof}
From this, we immediately get:
\begin{corollary}
    Every m-type cover is also an m-fpqc cover.
\end{corollary}
We are now ready to prove Theorem \ref{thm:sheaf_conditions}.
\begin{proof}[Proof of Theorem \ref{thm:sheaf_conditions}]
It is enough to show that any m-type covering in any of the big and small m-sites admits a refinement that is a covering in the Grothendieck topology generated by strict m-type covers and log blow-ups (see \cite[Section 2.3.5]{vistoli-2004-notes-on-grothendieck-topologies}).

Indeed, let $S$ be an fs log scheme and $\{U_i\rightarrow S\}_{i\in I}$ an m-type covering. By \cite[Theorem 5.4]{Toric-singularies}, we may replace $S$ with a log blow-up of $S$ that has a Zariski log structure. After passing to an open cover, we may further assume that $S$ is quasi-compact and admits a global chart. Proposition \ref{prop: quasi-compact is quasi-compact} now implies that there is a finite subcovering $\{U_{i'}\rightarrow S\}_{i'\in I'}$, and we may assume that the $U_{i'}$ are quasi-compact. Furthermore, according to \cite[Lemma 3.10]{NAKAYAMALOGETALECohomology2}, there is a log blow-up $S'\rightarrow S$ such that $\coprod_{i'\in I'}U_{i'}\times_{S}S'\rightarrow S'$ is exact and by Proposition \ref{m-type of Kummer is strict} therefore also strict. Hence, $\{U_{i'}\times_{S}S'\rightarrow U_{i'}\rightarrow S\}_{i'\in I'}$ is the desired refinement, which concludes the proof. 
\end{proof}
\begin{remark}
On the category $\mathsf{fsLogSch}/S$, the m-open (resp. m-\'etale, m-smooth, m-fppf) topology is equivalent to (i.e. has the same covering sieves as) the \emph{valuative Zariski (resp. \'etale, smooth, flat) topology} defined in \cite[Definition 2.23]{K-theory-of-log-schemes-I} and the \emph{dividing Zariski (resp. \'etale) topology} of the same type defined in \cite[Definition 3.1.5]{Triangulated-categories-of-logarithmic-motives-over-a-field}.  \end{remark}
\section{Subtleties regarding universal surjectivity}
In this section, we prove Theorem \ref{thm: shit be fucked} and show how to fix the problem that this theorem raises. 
\subsection{Proof of Theorem \ref{thm: shit be fucked} and consequences}\label{section where shit goes down}
As mentioned in Section \ref{subsection in introduction about universal surjectivity}, it suffices to show the following:
\begin{proposition}\label{proposition: the counterexample}
    The union of all log blow-ups of $\mathbb{A}^2_k$ restricted to the non-torus-fixed loci \begin{equation}\label{equation: da morphism is in da house}
        \bigsqcup_{\emptyset\neq I\subset \mathbb{N}^2}\left(\text{Bl}_I\mathbb{A}^2_k \right)_{\text{non-fix}}\rightarrow \mathbb{A}^2_k
    \end{equation}
    is weakly universally surjective. Moreover, weak universal surjectivity fails if one replaces the left-hand side by any finite subunion.
\end{proposition}
\begin{proof}
    Indeed, if one replaces the left-hand side of \eqref{equation: da morphism is in da house} by a union indexed by finitely many ideals $I_1,\ldots, I_n\subset \mathbb{N}^2$, then it follows that the base change along $\text{Bl}_I\mathbb{A}^2_k\rightarrow\mathbb{A}^2_k$ for $I = I_1\cdot\ldots\cdot I_n$ does not hit any of the torus-fixed points.
    
    For the first claim, it suffices to test surjectivity after base change along any log blow-up $\text{Bl}_I \mathbb{A}^2_k\rightarrow \mathbb{A}^2_k$. Indeed, it suffices to test surjectivity after base change along morphisms $g\colon T\rightarrow \mathbb{A}^2_k$ of fs log schemes with $T$ quasi-compact. By \cite[Lemma 3.10]{NAKAYAMALOGETALECohomology2}, there is a non-empty ideal $I\subset P$ such that the base change of $g$ along $\text{Bl}_I \mathbb{A}^2_k\rightarrow \mathbb{A}^2_k$ is exact. Since base change under exact morphisms preserves surjectivity \cite[(2.2.2)]{Log-etale-cohomology-I}, we may therefore assume that $g$ is a log blow-up along $I$ as claimed. 
    
    It is immediate that the image of the resulting base change contains $\left(\text{Bl}_I\mathbb{A}^2_k \right)_{\text{non-fix}}\subset\text{Bl}_I\mathbb{A}^2_k$. Now let $x_0\in \text{Bl}_I\mathbb{A}^2_k$ be any torus-fixed point. Now note that $x_0$ admits a non-torus-fixed preimage along the toric (hence log) blow-up $\text{Bl}_{x_0}\text{Bl}_I\mathbb{A}^2_k\rightarrow \text{Bl}_I\mathbb{A}^2_k$. Since the composite $\text{Bl}_{x_0}\text{Bl}_I\mathbb{A}^2_k\rightarrow \text{Bl}_I\mathbb{A}^2_k\rightarrow \mathbb{A}^2_k$ is by \cite[Corollary 4.11]{Toric-singularies} a log blow-up $\text{Bl}_J\mathbb{A}^2_k\rightarrow\mathbb{A}^2_k$ for some ideal $J\subset \mathbb{N}^2$, it follows that $x_0$ is in the image of the base change as desired.
\end{proof}
\begin{remark}\label{remark: topology not generated}
    In particular, contrary to \cite[Proposition 3.9]{NAKAYAMALOGETALECohomology2}, the log \'etale topology defined by using weakly universally surjective covers is not generated by log blow-ups and Kummer \'etale covers. Indeed, using Proposition \ref{prop: quasi-compact is quasi-compact}, Proposition \ref{prop: universally surjective examples} and the proof of \cite[Proposition 3.9]{NAKAYAMALOGETALECohomology2}, we see that Kummer \'etale covers and log blow-ups generate the log \'etale topology defined using universally surjective covers, which does not agree with the weak version as $\mathbb{A}^2_k$ is quasi-compact in one, but not the other.
\end{remark}

\subsection{Basic properties of universal surjectivity}
In this section, we collect several facts about universal surjectivity in the sense of Definition \ref{definition: universally surjective}.

\begin{lemma}\label{lemma: universal surjectivity def}
    Let $f\colon X\rightarrow Y$ be a morphism of fs log schemes. Then the following are equivalent:
    \begin{enumerate}
        \item\label{item: integral base change} $f$ is universally surjective.
        \item\label{item: saturated base change} For any saturated log scheme $T$ and any morphism $T\rightarrow Y$, the pullback $f_T\colon X\times_Y T\rightarrow T$ in the category of saturated log schemes is surjective.
        \item\label{item: log point base change} For any valuative log scheme $T$ with underlying scheme $\underline{T} = \text{Spec}(k)$ for $k$ an algebraically closed field, the base change $X\times_Y T$ in the category of integral (or saturated) log schemes is non-empty.
        \item\label{item: fval surjective}
        The morphism $f^{\text{val}}\colon X^{\text{val}}\rightarrow Y^{\text{val}}$ of Remark \ref{you can always blow up} is surjective.
    \end{enumerate}
\end{lemma}
\begin{remark}
    \begin{enumerate}
        \item By \cite[Lemma 3.1.9]{Log-geometry-beyond-fs} all the above fiber products exist. This is because $Y$ is an fs log scheme and thus any morphism from an integral log scheme to $Y$ admits charts, which one can use to construct the fiber product.
        \item See \cite[Proposition 2.2.4.2]{Molcho_Wise_2022} for a very similar statement.
    \end{enumerate}
\end{remark}
\begin{proof}
    For any integral log scheme $X$, the canonical morphism $X^{\text{sat}}\rightarrow X$ from its saturation is surjective (see e.g. \cite[Proposition 3.1.10]{Log-geometry-beyond-fs}). Hence, \eqref{item: integral base change} and \eqref{item: saturated base change} are equivalent; the integral and saturated versions of \eqref{item: log point base change} are also equivalent. Hence it suffices to show that \eqref{item: integral base change}, \eqref{item: log point base change} and \eqref{item: fval surjective} are equivalent. Moreover, it follows from Theorem \ref{thm: valuative space}\eqref{item: pull-back valuative} and Remark \ref{you can always blow up} that we may replace $X$ and $Y$ by strict \'etale covers and hence assume that they are Zariski. 

    Now note that $f$ is universally surjective if and only if for any morphism $T\rightarrow Y$ from an integral log point, we can find a commutative square
    \begin{equation}\label{diagram: fuck you}
    \begin{tikzcd}
        \exists \ T'\rar[dashed]\dar[dashed] &T\dar\\
        X \rar &Y,
    \end{tikzcd}
    \end{equation}
    where $T'$ is another log point. By Theorem \ref{thm: valuative space}, we have a surjection $T^{\text{val}}\twoheadrightarrow T$, hence we may assume that $T$ is valuative, which shows that \eqref{item: integral base change} and \eqref{item: log point base change} are equivalent. For the same reason, we can also ensure that $T'$ is valuative if it exists. By the universal property of the valuative space, we therefore see that \eqref{item: integral base change} implies \eqref{item: fval surjective}. More precisely, \eqref{item: fval surjective} is equivalent to \eqref{diagram: fuck you} being completeable if $T$ is a point of $Y^{\text{val}}$. To show that \eqref{item: fval surjective} implies \eqref{item: integral base change}, let $T_0\rightarrow Y$ be a morphism from a general valuative log point and $T_0\rightarrow T\rightarrow Y$ a factorization, where the last arrow comes from a point of $Y^{\text{val}}$ and thus fits into a square as in \eqref{diagram: fuck you}. We thus get a morphism $T'\times_T T_0\rightarrow X\times_Y T_0$, where we take fiber products in the category of integral log schemes. Indeed, note that $T'\times_T T_0$ exists as valuative log points admit canonical charts coming from their global sections, which enables one to use the standard construction of the fiber product. However, since $T$ is valuative, it follows from \cite[Proposition I.4.6.3(5)]{lecturesonlogarithmicalgebraicgeometry} that $T'\rightarrow T$ is integral and hence $\underline{T'\times_T T_0} = \underline{T'}\times_{\underline{T}}\underline{T_0}$ is non-empty, which implies that $X\times_Y T_0$ is also non-empty.
\end{proof}
We will now close this section by showing that universal surjectivity still admits the expected examples. 
\begin{proposition}\label{prop: universally surjective examples}
    If $f\colon X\rightarrow Y$ is a morphism of fs log schemes satisfying one of the following conditions:
    \begin{enumerate}
        \item\label{item: exact and surj is univ sur} exact and surjective,
        \item\label{item: qc and univ sur} quasi-compact and universally surjective in fs log schemes,
        \item\label{item: log blow is univ surj} a log blow-up,
    \end{enumerate}
    then $f$ is also universally surjective. Moreover, universally surjective morphisms are stable under composition and base change in fs log schemes.
\end{proposition}
\begin{proof}
    Stability under composition and base change follow directly from the definition. By Lemma \ref{lemma: universal surjectivity def}\eqref{item: log point base change} we can check the universal surjectivity of log blow-ups on valuative log points. Indeed, it follows from \cite[Lemma 2.5.5]{Log-geometry-beyond-fs} that any morphism from a valuative log point lifts along a log blow-up. This gives \eqref{item: log blow is univ surj}. For \eqref{item: qc and univ sur}, we use that universal surjectivity is strict \'etale local in $Y$, which allows us to assume that $Y$, and hence $X$, are quasi-compact and $Y$ admits a global chart. By \cite[Lemma 3.10]{NAKAYAMALOGETALECohomology2}, there is therefore a log blow-up $\widetilde{Y}\rightarrow Y$ so that $X\times_Y \widetilde{Y}\rightarrow \widetilde{Y}$ is exact and surjective. Since log blow-ups are universally surjective, this reduces \eqref{item: qc and univ sur} to \eqref{item: exact and surj is univ sur}. For \eqref{item: exact and surj is univ sur} we note that any log point $T$ in Lemma \ref{lemma: universal surjectivity def}\eqref{item: log point base change} has a canonical chart given by the global sections $Q=\Gamma(T,\mathcal{M}_T)$, which is exact at the unique point of $T$. Moreover, for any morphism $g\colon T\rightarrow Y$, every chart $P$ of $Y$ canonically fits into a chart $P\rightarrow Q$ for $g$. With these preparations in place, the proof of \cite[Proposition III.2.2.3]{lecturesonlogarithmicalgebraicgeometry} still works in our more general setting.
\end{proof}
\begin{remark}
Following \cite[Definition 4.1]{log-geometry-and-algebraic-stacks}, there is also a notion of $\mathcal{L}og$ surjectivity. Specifically, a morphism $f\colon X\rightarrow Y$ is \emph{$\mathcal{L}og$ surjective} if any base change of the (representable) morphism $\mathcal{L}og(f)\colon \mathcal{L}og_X\rightarrow \mathcal{L}og_Y$ along a scheme is surjective. Here $\mathcal{L}og_X$ denotes the stack of log structures of $X$. Indeed, one can show that $\mathcal{L}og$ surjectivity implies universal surjectivity. However, the converse fails.

For example, let $X=Y=\text{Spec}(\mathbb{N}\xrightarrow{0} \mathbb{C})$, and let $f\colon X\rightarrow Y$ be the exact morphism induced by $\mathbb{N}\xrightarrow{2\cdot}\mathbb{N}$. By Proposition \ref{prop: universally surjective examples}\eqref{item: exact and surj is univ sur}, $f$ is universally surjective. If $\mathcal{L}og(f)$ were surjective, there would exist a non-empty scheme $\underline{T}$ fitting into the diagram
\begin{center}
\begin{tikzcd}
    \underline{T}\rar\dar &\underline{Y}\dar\\
    \mathcal{L}og_X \rar["\mathcal{L}og(f)"] &\mathcal{L}og_Y,
\end{tikzcd}
\end{center}
where $\underline{Y}\rightarrow \mathcal{L}og_Y$ comes from the log structure on $Y$.
This is equivalent to the existence of a non-empty fs log scheme $T$ with a morphism $g\colon T\rightarrow X$ so that $f\circ g$ is strict. In particular, for any point $t\in T$, this implies that $\mathbb{N}\xrightarrow{2\cdot}\mathbb{N}\xrightarrow{g^*} \overline{\mathcal{M}}_{T,\overline{t}}$ is an isomorphism, which is impossible. Thus, $f$ is not $\mathcal{L}og$ surjective.
\end{remark}
\section{Proof of Theorem \ref{lemma-log-blow-up-integral}}
In this section we prove Theorem \ref{lemma-log-blow-up-integral}, which is needed in order to prove Proposition \ref{proposition: limit of topoi}.
In order to produce the desired log blow-up, we first need to introduce some new terminology.
\begin{definition} 
\begin{enumerate}
    \item A morphism $\varphi\colon P\rightarrow Q$ of sfs (sharp, fine and saturated) monoids is a \emph{coface map} if it induces an isomorphism $\overline{(\varphi^{-1}\{1\})^{-1}P}\xrightarrow{\cong} Q$. Recall here that $\overline{(\ldots)}$ denotes the quotient of a monoid by its subgroup of units. We also call $\varphi$ \emph{proper} if $\varphi$ is not an isomorphism.
    \item We denote by $\cC_{\text{sfs}}$ the category of sfs monoids with coface maps between them.
    \item A (full) subcategory $\cS\subset \cC_{\text{sfs}}$ is \emph{coface closed} if it contains all coface maps $P\rightarrow Q$ for $P\in \cS$.
    \item Let $\cS\subset\cC_{\text{sfs}}$ be a coface closed subcategory. We call a collection $(I_P)_{P\in\cS}$, where each $I_P\subset P$ is a non-empty ideal if for any coface map $\varphi\colon P\rightarrow Q$ in $\cS$ we have $\varphi(I_P)=I_Q$.
\end{enumerate}
\end{definition}
One can interpret compatible collections of ideals as a universal choice of log ideal on every algebraic stack with log structure. More precisely:
\begin{lemma}\label{lemma: compatible collections are universal ideals}
    Let $(I_P)_{P\in\cC_{\text{sfs}}}$ be a compatible collection of ideals on $\cC_{\text{sfs}}$. Then there is a unique assignment $(\cI_\cX)_{\cX}$ of coherent, nowhere empty log ideals $\cI_\cX$ on all algebraic stacks $\cX$ with fs log structure so that 
    \begin{enumerate}
        \item\label{item: strict pullback} For any strict morphism $f\colon \cX\rightarrow \cY$ we have $f^*\cI_\cY = \cI_\cX$.
        \item\label{item: if global chart then global chart} If $\cX$ admits a global chart $P\rightarrow\Gamma(\cX,\cM_\cX)$, then $(P\rightarrow \overline{P})^{-1}(I_{\overline{P}})$ is a chart for $\cI_\cX$.
    \end{enumerate}
\end{lemma}
\begin{proof}
    Properties \eqref{item: strict pullback} and \eqref{item: if global chart then global chart} determine the $\cI_\cX$ if they exist. To prove existence, we need to check that the ideals in \eqref{item: if global chart then global chart} glue, that is, $(Q\rightarrow \overline{Q})^{-1}I_{\overline{Q}}$ and $(Q'\rightarrow \overline{Q'})^{-1}I_{\overline{Q'}}$ induce the same ideal on any log scheme $X$ admitting both $Q$ and $Q'$ as global charts. However, this follows from the compatibility of $(I_P)_{P\in\cC_{\text{sfs}}}$ and the fact that the natural maps $\overline{Q}\rightarrow\overline{\cM}_{X,\overline{x}}\leftarrow \overline{Q'}$ are coface maps for any geometric point $\overline{x}$ of $X$.
\end{proof}
We will need the following auxiliary lemma.
\begin{lemma}\label{lemma: the ideal quotient is our lord and saviour}
    Let $P$ be an sfs monoid and $I,J\subset P$ ideals. Then the ideal quotient $(I : J) \coloneqq \{p\in P\mid p\cdot J\subset I\} $ has the following two properties:
    \begin{enumerate}
        \item\label{item: first piece of blubber} If $I=J\cdot K$ for some ideal $K\subset P$, then we also have $I=J\cdot (I:J)$.
        \item\label{item: second piece of blubber} For any coface map $\varphi\colon P\rightarrow Q$ we have $\varphi(I:J) = (\varphi(I):\varphi(J))$.
    \end{enumerate}
\end{lemma}
\begin{proof}
    Claim \eqref{item: first piece of blubber} follows from the definitions. To prove \eqref{item: second piece of blubber}, it suffices to show $(\varphi(I):\varphi(J))\subset \varphi(I:J)$. For this, let $q\in (\varphi(I):\varphi(J))$ and choose generators $j_1,\ldots, j_n$ of $J$. Since $\varphi$ is surjective, we have $\varphi(p)=q$ for some $p\in P$ and thus $\varphi(p\cdot j_l)\in \varphi(I)$ for all $l$. As $\varphi$ is a coface map, there exists $p'\in \varphi^{-1}\{1\}$ so that $p'\cdot p\cdot j_l\in I$ for all $l$ and thus $q = \varphi(p) = \varphi(p'\cdot p)\in \varphi(I:J)$ as desired. 
\end{proof}
The following lemma is the technical core of this section.
\begin{lemma}\label{lemma: extend bitchy the bitch boi}
    \begin{enumerate}
        \item\label{item: can extend compatible collections} If $\cS\subset \cC_{\text{sfs}}$ is a coface closed subcategory, then any compatible collection of ideals on $\cS$ can be extended to one on $\cC_{\text{sfs}}$.
        \item\label{item: compactible collections dominate finite collections} Let $P_1,\ldots,P_n$ be sfs monoids and $I_1,\ldots,I_n$ non-empty ideals in $P_i$'s, respectively. Then there is a compatible collection of ideals $(I'_P)_{P\in \cC_{\text{sfs}}}$ so that for every $i=1,\ldots,n$ we have $I'_{P_i} = I_{i}\cdot J_i$ for some ideal $J_i\subset P_i$.
    \end{enumerate}
\end{lemma}
\begin{proof}
    By induction on the dimension, we can prove \eqref{item: can extend compatible collections} by extending the collection one monoid at a time. More precisely, let $P$ be an sfs monoid and $(I_Q)_{Q\in \cS}$ a compatible system on the coface closed subcategory $\cS\subset\cC_{\text{sfs}}$ spanned by all of the proper cofaces of $P$. 
    We then claim that the ideal $I_P = \bigcap_{\pi\colon P\rightarrow Q}\pi^{-1}(I_Q)$, where the intersection runs over all proper coface maps, is automorphism-invariant and $\pi(I_P)=I_Q$ for any proper coface map $\pi\colon P\rightarrow Q$. Indeed, the only nontrivial part is to verify that we have $\pi(I_P) \supset I_Q$. For this, let $q\in I_Q$ and let $p\in P$ be any preimage of $q$ under $\pi$. For any possibly different coface map $\pi'\colon P\rightarrow Q'$ we have the common coface $\pi''\colon Q'\rightarrow \overline{\pi'(\pi^{-1}\{1\})^{-1} Q'}\eqqcolon Q''$ of $Q'$ and $Q''$ and since $\pi''\circ \pi'$ factors over $\pi$, we have by compatibility $\pi''(\pi'(p))\in I_{Q''} = \pi''(I_{Q'})$. Thus, we can find a $g_{\pi'}\in \pi^{-1}\{1\}$ so that $\pi'(p \cdot g_{\pi'})\in I_{Q'}$. It thus follows that $p' \coloneqq p\cdot\prod_{\pi'\colon P\rightarrow Q'} g_{Q'}$ is another preimage of $q$ under $\pi$ and for any proper coface map $\pi'\colon P\rightarrow Q'$ we have $\pi'(p') = \pi'(p\cdot g_{\pi'})\cdot \pi'(\prod_{\pi'\neq \pi''\colon P\rightarrow Q''} g_{\pi''})\in I_{Q'}$ since $I_{Q'}$ is an ideal. Thus we have $p'\in I_P$ as desired.
    
    We show \eqref{item: compactible collections dominate finite collections} by induction on $d = \text{max}_{i=1,\ldots,n} \text{dim}(P_i)$. We may assume that the $P_i$ are pairwise non-isomorphic and the $I_i$ are automorphism-invariant, since otherwise one can replace the ideals by appropriate products. By the induction hypothesis, there is a compatible system $(\widetilde{I}_P)_{P\in\cC_{\text{sfs}}}$ so that $\varphi(I_i)$ divides $\widetilde{I}_Q$ for any coface map $\varphi\colon P_i\rightarrow Q$ with $\text{dim}(Q)<d$. Letting $\cS\subset\cC_{\text{sfs}}$ be the coface closed subcategory generated by the $P_i$, we consider a new collection $(\widetilde{I}'_P)_{P\in \cS}$ of ideals defined by $\widetilde{I}'_P = \widetilde{I}_P$ if $\dim(P)<d$ and $\widetilde{I}'_{P_i} = I_i\cdot (\widetilde{I}_{P_i}:I_i)$ for $i=1,\ldots,n$ with $\dim(P_i)=d$. It follows from Lemma \ref{lemma: the ideal quotient is our lord and saviour} that the collection is compatible on $\cS$ and hence \eqref{item: can extend compatible collections} gives us the desired extension to $\cC_{\text{sfs}}$.
\end{proof}
We now combine Lemma \ref{lemma: compatible collections are universal ideals} and Lemma \ref{lemma: extend bitchy the bitch boi} to show that any local log ideal is dominated by a global log ideal. This will be our main tool in the proof of Theorem \ref{lemma-log-blow-up-integral}.
\begin{lemma}\label{lemma: can extend ideals}
    Let $f\colon \cX\rightarrow\cY$ be a strict morphism of algebraic stacks with fs log structures so that $\cX$ is quasi-compact and $\cI\subset \cM_{\cX}$ is a coherent nowhere empty log ideal. Then there is a coherent nowhere empty log ideal $\cJ\subset\cM_\cY$ so that $\cI$ locally divides $f^*\cJ$. In particular, we can factor $\text{Bl}_{f^*\cJ}\cX=\text{Bl}_\cJ\cY\times_\cY \cX\rightarrow \text{Bl}_\cI\cX\rightarrow \cX$. 
\end{lemma}
\begin{proof}
    Since $\cX$ is quasi-compact, we can cover it by finitely many charts for $\cI$. Lemma \ref{lemma: compatible collections are universal ideals} and Lemma \ref{lemma: extend bitchy the bitch boi} \eqref{item: compactible collections dominate finite collections} now imply the claim.
\end{proof}
\begin{remark}\label{remark: ideal might as well be globally charted}
    If, in addition, $\cX$ admits a global chart, then it follows from the above proof and Lemma \ref{lemma: compatible collections are universal ideals} that $\cJ$ does too. In particular, it follows from the case $f = \text{Id}_\cX$ that any coherent, nowhere empty log ideal on $\cX$ is dominated by a log ideal admitting a global chart.
\end{remark}
The following is a strengthening of \cite[Theorem 5.10]{Toric-singularies}, where the same statement is proved for log regular log schemes.
\begin{proposition}\label{proposition: free log structure}
    For any algebraic stack $\cX$ with fs log structure, there is a coherent nowhere empty log ideal $\cI\subset\cM_\cX$ so that $\text{Bl}_\cI \cX$ has free log structure. Moreover, we can choose $\cI$ so that for any strict morphism $f\colon \cY\rightarrow\cX$, where $\cY$ has free log structure, we have $f^*\cI = \cM_\cY$.
\end{proposition}
\begin{proof}
    We will now produce a compatible system of saturated ideals $(I_P)_{P\in \cC_{\text{sfs}}}$ so that $\text{Bl}_{\mathbb{Z}[I]} \mathbb{A}_P$ has free log structure for every sfs monoid $P$ and $I_P=P$ if $P$ is free. Lemma \ref{lemma: compatible collections are universal ideals} will then do the rest. Indeed, the collection of trivial (hence saturated) ideals on the coface closed subcategory of free monoids is compatible. To extend this collection to $\cC_{\text{sfs}}$, as in the proof of Lemma \ref{lemma: extend bitchy the bitch boi} \eqref{item: can extend compatible collections}, it suffices to show that for any sfs monoid $P$ and any compatible collection of saturated ideals $(I_Q)_{Q\in \cS}$ with free log blow-ups on the coface closed subcategory $\cS\subset\cC_{\text{sfs}}$ spanned by all of the proper cofaces of $P$, there is an automorphism-invariant saturated ideal $I\subset P$ with free log blow-up so that $\varphi(I) = I_Q$ for any proper coface map $\varphi\colon P\rightarrow Q$. However, this is precisely what is done in the proof of \cite[Theorem 5.10]{Toric-singularies} using the correspondence between saturated ideals and piecewise linear functions on the dual cone.
\end{proof}

We close this section by using the above machinery to prove Theorem \ref{lemma-log-blow-up-integral}.
\begin{proof}[Proof of Theorem \ref{lemma-log-blow-up-integral}]
    By \cite[Proposition I.4.7.5]{lecturesonlogarithmicalgebraicgeometry}, a morphism is integral if it is $\mathbb{Q}$-integral and its target has free log structure. By Proposition \ref{proposition: free log structure}, it therefore suffices to make $f$ $\mathbb{Q}$-integral after a log blow-up. First note that we can cover $f$ by squares of the shape
    \[
    \begin{tikzcd}
        U\dar \rar{\varphi} &V\dar \\
        \cX \rar["f"] &\cY
    \end{tikzcd}
    \]
    where $U$ and $V$ are affine and the vertical maps are strict and smooth. Moreover, we may assume that there is a coherent nowhere empty log ideal $\cI\subset \cM_V$ so that $\text{Bl}_{\varphi^*\cI}U\rightarrow \text{Bl}_\cI V $ is $\mathbb{Q}$-integral (see \cite[Theorem III.2.6.7]{lecturesonlogarithmicalgebraicgeometry} and \cite[Theorem 1.1]{kato2022integral}). Since $\cX$ is quasi-compact, we can find a finite collection $\{\varphi_i\colon U_i\rightarrow V_i\}_{i=1}^n$ of such morphisms so that $\coprod_i U_i\rightarrow \cX$ is surjective. Recall that $\mathbb{Q}$-integrality (in contrast to integrality) is stable under fs base change, hence, by Lemma \ref{lemma: can extend ideals}, we may assume that the ideals on the $V_i$ extend to $\cY$, where we can take their product, which is the desired ideal.
\end{proof}
\section{m-open topoi as limits}\label{lim of topii}
\begin{proposition}\label{proposition: limit of topoi}
Let $X$ be an fs log scheme. Then there is an equivalence of topoi 
\[\text{Shv}(X_{\text{mop}})=\underset{\widetilde{X}\rightarrow X}{\mathrm{lim}}\text{Shv}(\widetilde{X}{}_{\text{Zar}}),\]
where the limit runs over all log blow-ups of $X$.
\end{proposition}
\begin{proof}
This follows from Corollary \ref{corollary: m can be made non-m by log blowy} and an argument analogous to \cite[Proposition 5.6 (1)]{NAKAYAMALOGETALECohomology2}.
\end{proof}
\begin{definition}\label{definition of log dimension 2}
For any fs log scheme $X$, we define the \emph{log dimension}
$$\text{log-dim}(X)\coloneqq \sup_{\widetilde{X}\rightarrow X} \dim(\widetilde{X}),$$ 
where the supremum runs over all log blow-ups $\widetilde{X}\rightarrow X$.
\end{definition}
In Corollary \ref{corollary: both definitions of log dimension agree}, we will see that this is equivalent to Definition \ref{definition of log dimension 1}.
Using \cite[ Corollary 3.11, Theorem 3.12, and Theorem 3.17]{clausen2021hyperdescent}, we can immediately deduce the following generalization of Corollary \ref{cor: cohomology vanishing}:
\begin{proposition}
    For any fs log scheme $X$, the m-open $\infty$-topos $\text{Sh}(X_{\text{m-op}}; \text{An})$ of sheaves valued in anima has homotopy dimension $\leq \text{log-dim}(X)$.
\end{proposition}
To make the estimate useful, we now prove Theorem \ref{Theorem on log dimensions}.
\begin{proof}[Proof of Theorem \ref{Theorem on log dimensions}]
    Throughout this proof, we will write $r_x = \max\left(\text{rank} (\overline{\mathcal{M}}_{X,\overline{x}}^{gp})-1,0\right)$ for any $x\in X$.
    By \cite[0A4H]{stacks-project} we have $\text{dim}(X')=\dim(X)$ for any strict, surjective and \'etale morphism $f\colon X'\twoheadrightarrow X$. Using Lemma \ref{lemma: can extend ideals} and Definition \ref{definition of log dimension 2}, this further implies $\text{log-dim}(X')=\text{log-dim}(X)$; Moreover for any $x\in X$ we have $f^{-1}\overline{\{x\}} = \overline{f^{-1}\{x\}} = \bigcup_{f(y)=x}\overline{\{y\}}$, where the first equality uses that $f$ is open and the second holds since the union on the right is locally finite. This implies $\dim \overline{\{x\}} = \dim f^{-1}\overline{\{x\}} = \sup_{f(y) = x}\dim\overline{\{y\}}$ and since $r_y = r_x$ for any $y\in X'$ with $f(y) = x$, it follows that both sides of \eqref{equation: log-dim if loc noeth} have the same value for $X$ and $X'$. Hence Theorem \ref{Theorem on log dimensions} holds if and only if it holds for a strict \'etale cover and we may assume that $X$ has Zariski log structure. 

    Now let $x\in X$, and choose an open neighborhood $U\subset X$ of $x$ admitting a chart $P\rightarrow \Gamma(U,\cM_X|_U)$ that is neat at $x$. By Lemma \ref{lemma: can extend ideals}, Theorem \ref{Theorem on log dimensions} follows if we show that $\sup_{\emptyset\neq I\subset P} \dim (\text{Bl}_I U \cap \overline{\{x\}}) = \dim\overline{\{x\}} + r_x$, where the supremum runs over all non-empty ideals of $P$.  
    Indeed, let $I\subset P$ be a non-empty ideal and note that any fiber of $\text{Bl}_{\mathbb{Z}[I]}\mathbb{A}_P\rightarrow \mathbb{A}_P$ is also the fiber of $\text{Bl}_{k[I]}\text{Spec}(k[P])\rightarrow \text{Spec}(k[P])$ for some prime field $k$ and thus of dimension at most $\max\left(\dim k[P] - 1,0\right) = r_x$ (recall that $\dim k[P] = \text{rank}(P^{gp})$ by \cite[Proposition I.3.4.1]{lecturesonlogarithmicalgebraicgeometry}). Note also that this upper bound is an equality for fibers over the irrelevant set $Z_{\text{irr}} = V(I_{\text{irr}})\subset \mathbb{A}_P$ when $I = I_{\text{irr}} = P\setminus\{1\}$ is the irrelevant ideal. By \cite[02FY]{stacks-project}, the fibers of $\text{Bl}_I U\cap\overline{\{x\}}\rightarrow U\cap\overline{\{x\}}$ have the same dimension as those of $\text{Bl}_{\mathbb{Z}[I]}\mathbb{A}_P\rightarrow \mathbb{A}_P$ over the image of $U\cap\overline{\{x\}}$. By \cite[02FX, 0BAG]{stacks-project}, it therefore follows that $\dim (\text{Bl}_I U\cap\overline{\{x\}}) \leq \dim\overline{\{x\}} + r_x$ for all $I$. To show that this upper bound is sharp, we note that since $\overline{\{x\}}$ maps into $Z_{\text{irr}}$ and $\text{Bl}_{I_{\text{irr}}}\mathbb{A}_P\cap Z_{\text{irr}}\rightarrow Z_{\text{irr}}$ is flat of relative dimension $r_x$, it follows that $\text{Bl}_{I_{\text{irr}}} U\cap\overline{\{x\}}\rightarrow U\cap\overline{\{x\}}$ is also flat of relative dimension $r_x$ and we get $\dim (\text{Bl}_{I_{\text{irr}}} U\cap\overline{\{x\}}) = \dim\overline{\{x\}} + r_x$ by \cite[0AFE]{stacks-project}.
\end{proof}

\section{Sheaves on the valuative log space}\label{val log space section bla}
In this section, we prove the results claimed in Section \ref{section: intro valuative space}.
\subsection{Valuative log spaces}
We first need to define the so-called \emph{valuative log space} associated to an fs log scheme. This space is not a scheme, and its log structure is not quasi-coherent. We therefore begin with some generalities:
\begin{definition}
    A \emph{log locally ringed space} $X$ is a locally ringed space $\underline{X}$ equipped with a log structure $\mathcal{M}_X\rightarrow\mathcal{O}_X$ where $\mathcal{M}_X$ is a sheaf of integral monoids on $\underline{X}$. Morphisms of log locally ringed spaces are defined accordingly and we denote the resulting category by $\mathsf{LogLRS}$. We call $X$ \emph{valuative} if all stalks of $\mathcal{M}_X$ are valuative monoids.
\end{definition}
\begin{remark}\label{remark: limits in loglrs}
    \begin{enumerate}
        \item Any integral Zariski log scheme can be naturally regarded as a log locally ringed space.
        \item The category $\mathsf{LogLRS}$ has all cofiltered limits. Indeed, if $\{(X_i,\mathcal{O}_{X_i},\mathcal{M}_{X_i})\}_{i\in I}$ is a cofiltered system in $\mathsf{LogLRS}$, then the limit is given by
        \[
            \mathop{\mathrm{lim}}_{i\in I}  \, (X_i,\mathcal{O}_{X_i},\mathcal{M}_{X_i}) = (\underset{i\in I}{\mathrm{lim}}  \, X_i, \mathop{\mathrm{colim}}_{i\in I} \, \pi_i^{-1}\mathcal{O}_{X_i},  \underset{i\in I}{\mathrm{colim}} \, \pi_i^{-1}\mathcal{M}_{X_i}),
        \]
        where $\pi_i\colon \underset{j\in I}{\mathrm{lim}} \, X_j\rightarrow X_i$ is the $i$-th projection. Using \cite{LRSHasLimits}, one can show that $\mathsf{LogLRS}$ actually has arbitrary limits, but we will not need this fact.
    \end{enumerate}
\end{remark}
The following theorem is stated in \cite{kato_logarithmic_degeneration}, where its proof is omitted. We provide a proof here for the reader's convenience.
\begin{theorem}\cite[Theorem 1.3.1]{kato_logarithmic_degeneration}\label{thm: valuative space}
    For any integral Zariski log scheme $X$, there is a log locally ringed space $\pi\colon X^{\text{val}}\rightarrow X$ over $X$ with a valuative log structure, terminal among all valuative log locally ringed spaces over $X$. Moreover, we have:
    \begin{enumerate}
        \item\label{item: properties of valuative space} The morphism $\pi\colon X^{\text{val}}\rightarrow X$ is quasi-compact, quasi-separated, closed and surjective. Moreover, $X^{\text{val}}$ is locally spectral and if $X$ is quasi-compact and quasi-separated, then $X^{\text{val}}$ is spectral\footnote{Recall that a topological space $T$ is \emph{spectral} if it is sober, quasi-compact, quasi-separated and its quasi-compact opens form a base of its topology. A space that admits an open cover by spectral spaces is called locally spectral.}.
        \item\label{valuative space if chart} If $X$ has a global chart $P\rightarrow \Gamma(X,\mathcal{M}_X)$, then we have 
        \[X^{\text{val}} = \underset{I\subset P}{\mathrm{lim}}\,\text{Bl}_I X,\]
        where the limit ranges over the cofiltered poset of non-empty, finitely generated ideals $I\subset P$, ordered by $I\leq J$ whenever there is a finitely generated ideal $K\subset P$ so that $I=J\cdot K$.
        \item\label{valuative space is limit fs} If $X$ is fs, then
        \[
            X^{\text{val}} = \underset{\cI}{\mathrm{lim}}\,\text{Bl}_\cI X,
        \]
        where the limit runs over the cofiltered poset of coherent, nowhere empty log ideals $\cI\subset\cM_X$ with $\cI\leq \cJ$ if $\cJ$ locally divides $\cI$.
        \item\label{item: pull-back valuative} For any strict morphism of integral Zariski log schemes $f\colon X\rightarrow Y$, the following square is cartesian in $\mathsf{LogLRS}$
        \begin{equation}\label{equation: pullback valuative space}
            \begin{tikzcd}
                X^{\text{val}}\ar[dr, phantom, "\lrcorner", very near start]\dar["\pi_X"] \rar["f^{\text{val}}"] &Y^{\text{val}}\dar["\pi_Y"]\\
                X\rar["f"] &Y
            \end{tikzcd}
        \end{equation}
    \end{enumerate}
\end{theorem}
\begin{proof}
    For a strict morphism of integral log schemes $f\colon X\rightarrow Y$ such that $Y^{\text{val}}$ exists, it follows that the pullback $X\times_Y Y^{\text{val}} \rightarrow Y^{\text{val}}$ in $\mathsf{LRS}$ equipped with the log structure pulled back from $Y^{\text{val}}$ is valuative and satisfies the universal property of the valuative log space of $X$. To conclude the proof of \eqref{item: pull-back valuative}, it suffices to construct $X^{\text{val}}$ for any integral Zariski log scheme $X$. Indeed, it suffices to construct it for all members of an open cover of $X$ since the intersections will glue by uniqueness and \eqref{item: pull-back valuative}. As a result, we may assume that $X$ has a global chart $P\rightarrow \Gamma(X,\mathcal{M}_X)$, and it suffices to show \eqref{valuative space if chart}. 
    
    We now prove \eqref{valuative space if chart}. First, one can show that the universal property of the log blow-up also holds in $\mathsf{LogLRS}$; hence any valuative log locally ringed space mapping to $X$ admits a unique lifting to $X' ={\lim}_{I\subset P} \text{Bl}_I X$. It therefore suffices to show that $X'$ is valuative. To see that, let $x\in X'$ and $m\in \mathcal{M}_{X',x}^{gp}$. By Remark \ref{remark: limits in loglrs} we have 
    $\mathcal{M}_{X',x} = {\mathrm{colim}}_{I\subset P} \mathcal{M}_{\text{Bl}_I X,\pi_I(x)}$.
    Since group completion is a left adjoint, we also get $\mathcal{M}_{X',x}^{gp} = {\mathrm{colim}}_{I\subset P} \mathcal{M}_{\text{Bl}_I X,\pi_I(x)}^{gp}$, which is a filtered colimit and hence agrees with the colimit taken in the category of sets. Now let $I\subset P$ be a finitely generated ideal so that $m\in \mathcal{M}_{\text{Bl}_I X,\pi_I(x)}^{gp}$. By the construction of the log blow-up, there is a morphism $P^{gp}\rightarrow \mathcal{M}_{\text{Bl}_I X,\pi_I(x)}^{gp}$ coming from a local chart whose image along with $\mathcal{O}_{\text{Bl}_I X,\pi_I(x)}^\times$ generates $\mathcal{M}_{\text{Bl}_I X,\pi_I(x)}^{gp}$. As a result, there is an element $\frac{a}{b}\in P^{gp}$ that maps to $m$ modulo $\mathcal{O}_{\text{Bl}_I X,\pi_I(x)}^\times$. For $J = (a,b)\cdot I$ with $(a,b)\subset P$ the ideal generated by $a$ and $b$, this implies that $m\in \mathcal{M}_{\text{Bl}_J X,\pi_J(x)}$ or $m^{-1}\in \mathcal{M}_{\text{Bl}_J X,\pi_J(x)}$. This finally shows \eqref{valuative space if chart} and the existence of $X^{\text{val}}$. 
    
    For \eqref{valuative space is limit fs}, the universal property of the log blow-up induces a morphism $\eta_X\colon X^{\text{val}}\rightarrow \lim_\cI \text{Bl}_\cI X$. By Lemma \ref{lemma: can extend ideals} we furthermore have $\eta_X\times_X U = \eta_U$ for any strict quasi-compact open $U\subset X$ and hence it suffices to check that $\eta_X$ is an isomorphism if $X$ has a global chart, which follows from \eqref{valuative space if chart} and Remark \ref{remark: ideal might as well be globally charted}.
    
    To show the claims in \eqref{item: properties of valuative space}, we may assume that $X=\text{Spec}(P\rightarrow A)$, in which case the claim follows from Lemma \ref{lemma on cofiltered system} and the fact that log blow-ups are proper and surjective.
\end{proof}
\begin{remark}\label{you can always blow up}
    If $X$ is an fs \'etale log scheme, then by \cite[Theorem 5.4]{Toric-singularies} there is a log blow-up $\widetilde{X}\rightarrow X$ so that the log structure of $\widetilde{X}$ is Zariski. Thus, we may define $X^{\text{val}}\coloneqq \widetilde{X}^{\text{val}}$, which admits a projection $\pi\colon X^{\text{val}}\rightarrow \widetilde{X}\rightarrow X$ to $X$. This projection exists only as a morphism of locally ringed spaces. It is easy to verify that this is independent of the choice of $\widetilde{X}$ and Theorem \ref{thm: valuative space}\eqref{item: properties of valuative space}-\eqref{item: pull-back valuative} still hold in this setting. Using this construction, any morphism $f\colon X\rightarrow Y$ of fs \'etale log schemes also induces, functorially, a morphism $f^{\text{val}}\colon X^{\text{val}}\rightarrow Y^{\text{val}}$.
\end{remark}
The valuative log space of an fs log scheme $X$ can be viewed as an incarnation of its m-open site. More precisely, by Remark \ref{m-open local composition}, there is a morphism of sites 
\begin{align*}
    \Phi\colon X_{\text{mop}}&\longrightarrow X^{\text{val}}_{\text{Zar}}   \\
    \left(T\rightarrow X\right) &\mapsto \left( T^{\text{val}}\rightarrow X^{\text{val}}\right),
\end{align*} 
where for any topological space $Y$ we let $Y_{\text{Zar}}\subset \mathsf{Top}/Y$ be the site of topological spaces with local homeomorphisms $T\rightarrow Y$ and coverings $\{T_i\rightarrow T\}_{i\in I}$ consisting of jointly surjective maps over $Y$. Note that there is a canonical equivalence $\text{Shv}(Y_{\text{Zar}}) = \text{Shv}(Y)$. We are now ready to give a more precise statement and proof of Theorem \ref{thm: sheaves on val is m}. \vspace{0.3cm}\\
\textbf{Theorem \ref{thm: sheaves on val is m}.} The induced pushforward 
\[\Phi_*\colon \text{Shv}(X^{\text{val}})\xlongrightarrow{\cong} \text{Shv}(X_{\text{mop}})\] 
is an equivalence.
\begin{proof}[Proof of Theorem \ref{thm: sheaves on val is m}]
    By Theorem \ref{thm: valuative space}\eqref{item: properties of valuative space}, $X^{\text{val}}$ is locally spectral and hence its set of quasi-compact opens $\text{qcop}(X^{\text{val}})$ forms a base for its topology. It follows from Theorem \ref{thm: valuative space}\eqref{item: properties of valuative space},\eqref{valuative space is limit fs} that $\text{qcop}(X^{\text{val}}) = \text{colim}_{\widetilde{X}\rightarrow X} \text{qcop}(\widetilde{X})$, where the colimit runs over all log blow-ups of $X$. Since these log blow-ups are also locally spectral, we obtain $\text{Shv}(X^{\text{val}}) = \lim_{\widetilde{X}\rightarrow X} \text{Shv}(\widetilde{X}{}_{\text{Zar}})$, which together with Proposition \ref{proposition: limit of topoi} implies the claim. 
\end{proof}
Finally, we show that the two seemingly different definitions of logarithmic dimension given in Definition \ref{definition of log dimension 1} and Definition \ref{definition of log dimension 2} actually agree. We will deduce this from the following technical Lemmas.
\begin{lemma}\label{lemma on cofiltered system}
Let $\{X_i\}_{i\in I}$ be a cofiltered system of (locally) spectral spaces with quasi-compact (and quasi-separated) transition maps $\pi_{i,j}\colon X_i\rightarrow X_j$. Let $X={\mathrm{lim}}_{i\in I}\,X_i$ be its limit, and let $\pi_i\colon X\rightarrow X_i$ the projections. Then:
    \begin{enumerate}
        \item\label{item-quasi-compact} $X$ is also (locally) spectral and the $\pi_i$ are quasi-compact (and quasi-separated).
        \item\label{item-closed-surjective} If all of the $\pi_{i,j}$ are closed and surjective, then the $\pi_i$ are closed and surjective.
    \end{enumerate}
\end{lemma}
\begin{proof}
Since any quasi-compact and quasi-separated locally spectral space is spectral, the locally spectral version of the above statement easily reduces to the spectral version.
If the $X_i$ are spectral and the $\pi_{i,j}$ are quasi-compact, then $X$ is also spectral by \cite[0A2Z]{stacks-project} and the projections $\pi_i\colon X\rightarrow X_i$ are quasi-compact, which concludes the proof of \eqref{item-quasi-compact}. 

For \eqref{item-closed-surjective}, we note that if all transition maps are surjective, then by \cite[0A2W(2),0902,0A2S(1)]{stacks-project} the $\pi_i$ are also surjective. For closedness of the $\pi_i$, let $C\subset X$ be a closed subset and $x\in X_i\setminus\pi_i(C)$. Then we have $\pi_i^{-1}\{x\}\cap C = \emptyset$ and hence we can cover $\pi_i^{-1}\{x\}$ by open neighborhoods pulled back from some $X_j$ that are disjoint from $C$. By \cite[0A2S(2)]{stacks-project}, $\pi_i^{-1}\{x\}$ is quasi-compact, hence we may find $j\leq i$ and $U\subset X_j$ disjoint from $\pi_j(C)$ so that $\pi_i^{-1}\{x\}\subset \pi_j^{-1} (U)$. Now letting $C' = X_j\setminus U\subset X_j$, we get that $\pi_i(C)\subset \pi_{i}(\pi_{j}^{-1}(C')) = \pi_{j,i}(C')\not\ni x$, where we used surjectivity of $\pi_j$ in the equality. Since $\pi_{j,i}$ is closed, this implies that $x\not\in \overline{\pi_i(C)}$. As a result, $\pi_i(C)$ is closed and we are done.
\end{proof}

\begin{lemma}\label{lemma: dimension cofiltered}
    Let $\{X_i\}_{i\in I}$ be a cofiltered system of topological spaces and $X={\mathrm{lim}}_{i\in I}\,X_i$ its limit. Then we have
    \[
        \dim(X) \leq \sup_{i\in I} \dim (X_i)
    \]
    and equality if the $X_i$ are locally spectral and the transition maps $\pi_{i,j}\colon X_i\rightarrow X_j$ are surjective, quasi-compact, quasi-separated and closed.
\end{lemma}
\begin{proof}
    Let $Z_1\subsetneq Z_2\subset X$ be a proper inclusion of irreducible closed subsets. Since the opens pulled back from the projections $\pi_i\colon X\rightarrow X_i$ form a basis for the topology of $X$, we can choose $x\in Z_2\setminus Z_1$ and an open neighborhood $U\subset X_i$ of $\pi_i(x)$ for some $i\in I$ so that $Z_1\cap \pi_i^{-1}(U)=\emptyset$. This shows that $\overline{\pi_j(Z_1)}\subsetneq \overline{\pi_j(Z_2)}$ for all $j\leq i$ and hence any finite chain of irreducible closed subsets in $X$ descends to one of the same length on some $X_i$. This proves the desired inequality. Combining Lemma \ref{lemma on cofiltered system} and \cite[02JF]{stacks-project}, we get the desired equality of dimensions.
\end{proof}
\begin{corollary}\label{corollary: both definitions of log dimension agree}
    For any fs log scheme $X$, we have 
    \[
        \dim(X^{\text{val}}) = \underset{\widetilde{X}\rightarrow X}{\sup} \dim(\widetilde{X}),
    \]
    where the supremum runs over all log blow-ups of $X$.
\end{corollary}
\begin{proof}
    This follows from Theorem \ref{thm: valuative space}\eqref{valuative space is limit fs} and Lemma \ref{lemma: dimension cofiltered}.
\end{proof}

\printbibliography
\end{document}